\documentclass{article}
\usepackage{amsmath}
\usepackage{amsfonts}
\usepackage{amssymb}
\usepackage{amsthm }
\usepackage[numbers]{natbib}
\usepackage{graphicx}
\usepackage{cleveref}
\usepackage{enumerate}
\usepackage{epstopdf}
\usepackage{fullpage}
\usepackage{color}
\usepackage{caption}

\newtheorem{Thm}{Theorem}[section]

\newtheorem{Con}[Thm]{Conjecture}
\newtheorem{Lem}[Thm]{Lemma}
\newtheorem{Cor}[Thm]{Corollary}
\newtheorem{Pro}[Thm]{Proposition}
\newtheorem{Ass}{Assumption}

\makeatletter
\def\blfootnote{\xdef\@thefnmark{}\@footnotetext}
\makeatother

\theoremstyle{definition}

\theoremstyle{remark}
\newtheorem*{Rem}{\bf{Remark}}

\newcommand{\wh}{\widehat}
 \def\floor#1{\left\lfloor #1 \right\rfloor}
\newcommand{\EqD}{\overset{d}{=}}

\newcommand{\ConvFDD}{\overset{f.d.d.}{\longrightarrow}}

\newcommand{\cl}{\mathcal}

\newcommand{\bb}{\mathbb}

\newcommand{\E }{\mathbb{E}}
\newcommand{\mbf}{\boldsymbol}
\newcommand{\cov}{\mathrm{Cov}}
\newcommand{\Var}{\mathrm{Var}}

\newcommand{\RV}{\mathrm{RV}}
\newcommand{\lf}{\lfloor}
\newcommand{\rf}{\rfloor}

\newcommand{\supp}{\mathrm{supp}}

\title{ Limit Theorems for Conservative Flows on  Multiple Stochastic   Integrals}
\author{Shuyang Bai\\bsy9142@uga.edu}

\begin{document}

 \maketitle

\begin{abstract}
We consider a stationary  sequence $(X_n)$   constructed by a  multiple stochastic integral  and  an   infinite-measure  conservative dynamical system. The random measure defining the multiple   integral is non-Gaussian, infinitely divisible and has a finite variance. Some additional assumptions on the dynamical system give rise to  a parameter $\beta\in(0,1)$ quantifying the conservativity of the   system. This parameter $\beta$ together with the order of the integral determines the decay rate of the covariance of $(X_n)$. The goal of the paper is to establish limit theorems for the partial sum process of $(X_n)$. We obtain a central limit theorem  with Brownian motion as limit when the covariance decays fast enough, as well as a non-central limit theorem  with fractional Brownian motion or Rosenblatt process  as limit  when the covariance decays slow enough.

\end{abstract}

\textbf{MCS Classification (2010)}:   		60F17.

\section{Introduction}\label{Sec:Intro}

For a stationary random sequence $(X_k)$ with finite variance,  the notion \emph{long-range dependence} (or \emph{long memory}) is typically associated with  a slow power-law decay in the covariance $\cov(X_k,X_0)$ as $k\rightarrow\infty$. An  important perspective towards long-range dependence is to focus on its implication on limit theorems  (cf.\ \cite{samorodnitsky:2016:stochastic}).   In particular, if one is interested in a limit theorem for     sums, it is well-known   that when $(X_k)$ has a finite variance marginally and is  weakly dependent,   one   expects  the  following  \emph{central limit theorem}: as $n\rightarrow\infty$,
\begin{equation}\label{eq:CLT}
\left(\frac{1}{n^{1/2}}\sum_{k=1}^{\floor{nt}} X_k\right)_{t\ge 0} \Rightarrow \left(\sigma B(t)\right)_{t\ge 0},
\end{equation}
where $\Rightarrow$ stands for a suitable mode of weak convergence, e.g., weak convergence in the  Skorokhod space   $D[0,1]$,      $(B(t))_{t\ge 0}$ is the standard Brownian motion, and $\sigma^2 =\sum_k \cov(X_k,X_0)$ with the summability  often implied by the weak dependence condition imposed. On the other hand, when $(X_k)$ has long-range dependence where   $\cov(X_k,X_0)$   fails to be summable, the normalization $n^{1/2}$ in \eqref{eq:CLT} is not strong enough to stablize the   sum, and hence the central limit theorem \eqref{eq:CLT}  no longer holds. In fact in this case, one   anticipates instead a limit theorem of the form:
\begin{equation}\label{eq:nCLT}
\left(\frac{1}{n^{H}L(n)}\sum_{k=1}^{\floor{nt}} X_k\right)_{t\ge 0} \Rightarrow \left(Z(t)\right)_{t\ge 0},
\end{equation}
where $L$ is a slowly varying function (e.g. a logarithm, cf.\ \cite{bingham:goldie:teugels:1989:regular}), $H\in (1/2,1)$ is the so-called Hurst index, and $Z(t)$ is a $H$-self-similar    (i.e., $(Z(ct))_{t\ge 0}\EqD (c^H Z(t))_{t\ge 0}$, $c>0$) with stationary increments which can be Gaussian or non-Gaussian. According to \cite{samorodnitsky:2016:stochastic}, the phase transition     from \eqref{eq:CLT} to \eqref{eq:nCLT} can be viewed as an indication of the long-range dependence in $(X_k)$.

An aspect of the limit theorem in \eqref{eq:nCLT}  fundamentally different from \eqref{eq:CLT} is the absence of universality of the limit. In \eqref{eq:CLT} the limit process is always  a Brownian motion regardless of the distribution of $(X_k)$. In contrast, the limit $Z(t)$ in  \eqref{eq:nCLT} has many possibilities and often reflects some details in   $(X_k)$. A limit theorem \eqref{eq:nCLT} under  long-range dependence is often termed as a \emph{non-central limit theorem}.
A celebrated family of processes which are often found as limits  in \eqref{eq:nCLT}  are the Hermite processes  \cite{taqqu:1979:convergence,dobrushin:major:1979:non}. A standardized Hermite process can be represented as:
\begin{equation}\label{eq:Herm proc}
Z_{p,\beta}(t)= a_{p,\beta} \int_{\bb{R}^p}'  \left(\int_0^t  \prod_{j=1}^p (s-x_j)^{\beta/2-1}_+ ds \right) W(dx_1)\ldots W(dx_p),\quad t\ge 0,
\end{equation}
where $\int_{\bb{R}^p}' \cdot \ W(dx_1)\ldots W(dx_p)$ denotes the multiple Wiener-It\^o integral with $W$ being a Gaussian random measure with Lebesgue control measure, 
\begin{equation}\label{eq:beta range}
1-\frac{1}{p}<\beta<1,
\end{equation}
 and 
 \[
 a_{p,\beta}=\left( \frac{(1-p(1-\beta)/2)(1-p(1-\beta))}{p! \mathrm{B}(\beta/2,1-\beta)^p} 
\right)^{1/2} 
 \] is a constant which makes $\Var[Z_{p,\beta}(1)]=1$ (see, e.g., \cite[Proposition 4.4.2]{pipiras:2017:long}). Here $\mathrm{B}(x,y)=\int_0^1 z^{x-1}(1-z)^{y-1} dz$, $x,y>0$, denotes the beta function. The process $Z_{p,\beta}(t)$
has Hurst index $H=1-p(1-\beta)/2$, which is also known as the fractional Brownian motion when $p=1$ and the Rosenblatt process when $p=2$.

Recently, there has been some interest in studying limit theorems for a stationary  sequence $(X_k)$ constructed through embedding an infinite-measure dynamical system  in  a (multiple) stochastic integral. Here we only minimally describe the necessary ingredients and leave  the precise definitions  to Section \ref{Sec:prelim}.
 Suppose $(E,\cl{E},\mu)$ is a measure space where $\mu$ is a $\sigma$-finite infinite measure. Fix a subspace $A\in \cl{E}$ satisfying $\mu(A)<\infty$. Let $T:E\rightarrow E$ be a $\mu$-preserving transform.   Some additional ergodic-theoretic assumptions  need  to be imposed on the dynamical system $(E,\cl{E},\mu,T)$ and the set $A$. They give rise to a parameter $\beta\in (0,1)$ (it will   be identified with $\beta$  in \eqref{eq:beta range} in the context of limit theorems),  which loosely speaking, characterizes the frequency   of  the visits of the flow $\{T^n\}$ to $A$.
Suppose $M$ is a   symmetric infinitely-divisible  random measure with control measure $\mu$. For a suitable measurable function $f:E^p\rightarrow \bb{R}$ which has support within $A^p$, we can define a stationary sequence by
\begin{equation}\label{eq:X_k intro}
X_k=\int_{E^p}' f(T^k x_1 ,\ldots, T^k x_p)\  M(dx_1)\ldots M(dx_p),\quad k\in \bb{Z}_+,
\end{equation}
where the prime $'$ indicates the exclusion of the diagonals $x_i=x_j$, $i\neq j$, in the multiple integral. The strength of  the dependence  in   $(X_k)$   is controlled by the parameter $\beta$ and the order $p$, whereas the heaviness of the tails of $X_k$ is controlled by the   random measure $M$. See \cite{szulga:1991:multiple} for a general theory of multiple integrals with respect to  a symmetric infinitely-divisible random measure. We shall also provide a brief introduction  in Section \ref{sec:mult int} below.

Initiated from the work of    \citet{owada:samorodnitsky:2015:functional}, the focus had been mainly on the case where $p=1$ and $M$ (and thus $X_k$) has an infinite variance 
\cite{owada:2015:maxima, jung:2017:functional,  samorodnitsky:2016:stochastic,  owada:2016:limit, lacaux:2016:time, samorodnitsky:2019:extremal,chen:2020:extreme, chen:2020:extremal}.  Limit theorems for  sums of $(X_k)$ involving a general integral order $p$  were recently considered when $M$ is Gaussian \cite{bai:2020:limit}  and  when $M$ is infinitely divisible without a Gaussian component and has an infinite variance \cite{bai:2020:functional}. See Table \ref{tab:ref} for a crude classification of the existing results on limit theorems for sums.  

\begin{table}[h]
\centering
\begin{tabular}{|c|c|c|}
\hline 
Assumption on $M$  & $p=1$ & $p\ge 2$   \\ 
\hline 
 Infinite variance &  \cite{owada:samorodnitsky:2015:functional, jung:2017:functional}  &   \cite{bai:2020:functional} (only when $\beta>1-1/p$)  \\ 
\hline 
 Gaussian &  \cite{bai:2020:limit}  & \cite{bai:2020:limit}  \\
\hline 
Finite variance non-Gaussian & \cite{samorodnitsky:2016:stochastic}, current work &  current work (when $\beta>1-1/p$ or $p=2,\beta>1/2$)
\\ \hline
\end{tabular}
\caption{Summary of existing literature on limit theorems for sums regarding the model \eqref{eq:X_k intro}. The classification presented here is crude and does not reflect many  special assumptions and restrictions in these studies.}\label{tab:ref}
\end{table}

In this paper, we shall consider limit theorems for sums of $(X_k)$ in \eqref{eq:X_k intro} where $M$ is a finite-variance symmetric infinitely divisible  random measure without a Gaussian component.  We shall establish  a central limit theorem  when $\beta<1-1/p$ and $p\in \bb{Z}_+$ with a Brownian motion as limit, as well as a non-central limit theorem  with a Hermite process  in \eqref{eq:Herm proc} as limit  when   $\beta>1-1/p$ and $p=1$ or $2$.  We  note that a non-central limit theorem when $p=1$ has been considered in \cite[Section 9.4]{samorodnitsky:2016:stochastic} for  a type of dynamical systems constructed by null-recurrent Markov chains, for which  the proof  relies on  the infinite divisibility of the single integral and thus do not seem to apply to the case where $p= 2$. On the other hand, while the results are analogs of some of those in \cite{bai:2020:limit}; however, due to the absence of Gaussianity,   the proof techniques are different here. Our proof strategy involves first an approximation of $M$ by a compound Poisson random measure with finite moments, and  then a subtle execution  of the method of moments.   The case $\beta>1-1/p$ and $p\ge 3$, however, cannot be concluded using the method of moments, and is left as a conjecture.

The paper is organized as follows. Section \ref{Sec:prelim} prepares the background on infinite ergodic theory and multiple stochastic integrals.   Section \ref{Sec:main} contains the main results. Section \ref{Sec:proofs} includes the proofs of the main results.

\section{Preliminaries}\label{Sec:prelim}

First, we  address some notation which will be used throughout the paper. For two positive sequences $a_n$ and $b_n$, we write  $a_n\sim b_n$ to mean  $\lim_n a_n/b_n=1$.   For $a\in \bb{R}$, the notion  $\RV_\infty(a)$ denotes  the class of  positive functions defined on $\bb{Z}_+$ or $\bb{R}_+$, which are  regularly varying  with index $a$ at  infinity  (cf.\ \citet{bingham:goldie:teugels:1989:regular}).  Write $\lf x \rf=\sup\{n\in\bb{Z}:~n\le x\}$, $x\in \bb{R}$. For a function   $f:E\rightarrow \bb{R}$, we set  $\supp(f)=\{x\in E:~ f(x)\neq 0\}$.  The gamma function   is
$\Gamma(x)=\int_{0}^\infty u^{x-1} e^{-u} du$, $x>0$. For a measure space $(E,\cl{E},\mu)$, the notation $(E^p,\cl{E}^p,\mu^p)$ denotes its $p$-product  measure space, $p\in \bb{Z}_+$. For an integrable or non-negative measurable function $f$ defined on $(E,\cl{E},\mu)$, we use $\mu(f)$ to denote its integral $\int_E   f d\mu$.  For a finite ordered sequence   $I=(i_1,\ldots,i_p)\in \bb{Z}_+^p$, $i_1<\ldots<i_p$,   and   $(a_i)_{i\in \bb{Z}_+}$, we write 
$
a_I=(a_{i_1},\ldots,a_{i_p}).
$ 
This ordered sequence $I$ is at times   treated as a subset   of $\bb{Z}_+$ as well.
\subsection{Background from infinite ergodic theory}\label{sec:ergodic}

We  introduce some necessary backgrounds from the infinite ergodic theory, for which the main reference is
  \cite{aaronson:1997:introduction}.

Let $(E,\cl{E},\mu)$ be a measure space where $\mu$ is a $\sigma$-finite measure with $\mu(E)=\infty$. Suppose that $T: E\rightarrow E$ is a measure-preserving transform, namely, $T$ is measurable and $\mu(T^{-1}B)=\mu(B)$ for any $B\in \cl{E}$. We shall always make the following two basic assumptions:
\begin{itemize}
 \item  $T$ is \emph{ergodic}, namely, $T^{-1} B= B$ mod $\mu$ implies either $\mu(B)=0$ or $\mu(B^c)=0$; 
 \item $T$ is \emph{conservative}, namely, for any $B\in \cl{E}$ with $\mu(B)>0$, we have
$
\sum_{k=1}^\infty 1_B(T^k x)=\infty \ \text{ for a.e. $x\in B$}.
$
\end{itemize}
These two    assumptions  combined  are equivalent to the following statement (\cite[Proposition 1.2.2]{aaronson:1997:introduction}):
   for any  $B\in \cl{E}$   with $\mu(B)>0$, we have
\begin{equation}\label{eq:erg and cons}
\sum_{k=1}^\infty 1_{B}(T^k x)=\infty  \quad \text{ for a.e. }x\in E.
\end{equation}

The \emph{dual} (or \emph{Perron-Frobenius} or \emph{transfer}) operator $\wh{T}$ of $T$  is defined as   
\begin{equation*}
\wh{T}: L^1(\mu)\rightarrow L^1(\mu),\quad  \wh{T} f = \frac{d(\mu_f  \circ T^{-1}) }{d\mu},
\end{equation*}
where the signed measure $\mu_f(B):=\int_B f d\mu$, $B\in \cl{E}$.  The dual operator $\wh{T}$ is characterized by the dual property:
\begin{equation}\label{eq:dual}
\int_E f \cdot (g\circ T) d\mu=\int_E  (\wh{T}f) \cdot g d\mu
\end{equation}
for any  $f\in L^1(\mu)$ and $g\in L^{\infty}(\mu)$.
It is useful to note the probabilistic interpretation of $\wh{T}$: if $f$ is the  density with respect to $\mu$ of  a random element $X$ taking value in $E$,  then $\wh{T} f$ is the   density with respect to $\mu$ of the transformed random element $TX$. To see this, one can take $g$ to be an indicator function $1_B$, $B\in \cl{E}$, and then the dual property above becomes $P(TX\in B)=\int_{T^{-1}B} f d\mu=\int_B \wh{T}f d\mu $.  In addition, the characterization   in (\ref{eq:erg and cons}) has an equivalent  formulation using the dual operator (Proposition 1.3.2 \citet{aaronson:1997:introduction}):    for any non-negative $f\in L^1(E,\cl{E},\mu)$ satisfying $\mu(f)=\int_E f d\mu >0$, we have 
\begin{equation}\label{eq:erg and cons dual}
\sum_{k=1}^\infty \wh{T}^k f =\infty  \quad a.e..
\end{equation}

The following assumption was proposed in \citet{bai:2020:functional} combining the ideas from \citet{krickeberg:1967:strong} and literature on infinite ergodic theory such as  \citet{kessebohmer:2007:limit}. A similar framework  was also used in \citet{bai:2020:limit}.
 We shall adopt the following convention throughout: \emph{any function defined on a subspace 
(e.g.~$A$)  will be  extended  to the full space (e.g.~$E$) by assuming zero value outside the subspace, whenever  necessary}.

\begin{Ass}\label{ass:dyn} Let the dynamical system $(E,\cl{E},\mu,T)$ be as above. 
There exists  $A\in\cl{E}$ with $\mu(A)\in (0,\infty)$, so that  $A$ is a Polish space  with $\cl{E}_A:=\cl{E}\cap A$ being its  Borel $\sigma$-field.  In addition,  there exists a positive  rate sequence  $ (b_n)$ is regularly varying:  
\begin{equation}\label{eq:bn_RV}
(b_n)\in \RV_{\infty}(1-\beta),\quad \beta\in (0,1),
\end{equation}
  and
\begin{equation}\label{eq:uniform ret}
\lim_n b_n \wh{T}^n  g(x)=  \mu(g)=\int_E g d\mu=\int_A g d\mu \quad \text{uniformly for a.e. }x\in A 
\end{equation}
for  
 all  bounded  and $\mu$-a.e.\   continuous $g$   on $A$ (i.e., the discontinuity set of $g$ has $\mu$ measure zero),  where  for simplicity we still use $\mu$ to denote its restriction to $\cl{E}_A$. (The last equality in \eqref{eq:uniform ret}  is due to that $g$ has  its support within $A$ by the convention made before Assumption \ref{ass:dyn}.) 
\end{Ass}
Assumption \ref{ass:dyn}  implies a rate  of order $\RV_\infty(\beta)$ for  the divergence in \eqref{eq:erg and cons dual} over the subspace $A$, and hence  $\beta$ introduced in Assumption \ref{ass:dyn}  may be viewed as a parameter quantifying the conservativity of the system. The   relation \eqref{eq:uniform ret} in general cannot be extended to an arbitrary integrable function $g$ on $A$ due to the existence of weakly wandering sets \cite{hajian:1964:weakly}.   Under the assumptions imposed so far, the sequence $(b_n)$ is also related to the \emph{wandering sequence}:  
\begin{equation}\label{eq:w_n}
w_n   =\mu\left(\bigcup_{k=1}^n T^{-k} A\right) ,
\end{equation} 
through the following relation:  
\begin{equation}\label{eq:b_n w_n}
b_n\sim \Gamma(\beta)\Gamma(2-\beta) w_n 
\end{equation}
as $n\rightarrow\infty$. The relation \eqref{eq:b_n w_n} can be found in Proposition 3.1 of \citet{kessebohmer:2007:limit}. In particular, as with the current paper, ergodicity and conservativity   have been assumed throughout \citet{kessebohmer:2007:limit}, and Assumption 1 above is implied by their ``uniform return'' assumption imposed in their Definition 3.1.
\begin{Rem}
Specific examples of dynamical system $(E,A,\cl{E},\mu,T)$    satisfying Assumption 1, which involve interval maps with indifferent fixed points and  null-recurrent Markov chains,  can be found in \cite[Section 4.3]{bai:2020:functional}.  We shall omit repeating these concrete examples in this paper.  See also \cite{kessebohmer:2007:limit, gouezel:2011:correlation, melbourne:2012:operator} for more examples and theories related to Assumption \ref{ass:dyn}.
\end{Rem}

Assumption \ref{ass:dyn} implies the following ``mixing-type'' relation which essentially determines the covariance decay of $(X_k)$ in \eqref{eq:X_k intro}.  
\begin{Lem}\label{Lem:Cov Asymp}
Suppose $f_1$ and $f_2$ are   bounded and $\mu^{p}$-a.e.\ continuous functions on $A^p$ (for simplicity we still use $\mu^p$ to denote its restriction to $\cl{E}_A^p$) . Then 
\[
 \int_{E^p}f_1(T^nx_1,\ldots,T^nx_p) f_2(x_1,\ldots,x_p) \mu(dx_1)\ldots \mu(dx_p)  = \mu^{p}((f_1\circ T_p^n)\cdot f_2)  \sim   b_n^{-p} \mu^{ p}(f_1) \mu^{p}(f_2)
\]
as $n\rightarrow\infty$, where \[T_p:=T\times \ldots \times T\] is the Cartesian product transform.
\end{Lem}
\begin{proof}
For any $B_1,B_2\in \cl{E}_A$ such that $1_{B_1}$ and $1_{B_2}$ are $\mu$-a.e.\ continuous functions on $A$ (namely,  $\mu(\partial B_1)=\mu (\partial B_2)=0$, where $\partial B$ denotes the boundary of a set $B\subset A$), using the dual property of $\wh{T}$ in \eqref{eq:dual},  the    uniform convergence in \eqref{eq:uniform ret} implies that
\begin{equation}\label{eq:unif ret to mix}
\mu(B_1 \cap T^{-n} B_2)=\int_E 1_{B_1}  \cdot  (1_{B_2} \circ T^n) d\mu=\int_A \left(\wh{T}^n 1_{B_1}\right)   \cdot 1_{B_2}   d\mu \sim b_n^{-1} \mu(B_1) \mu(B_2)
\end{equation}
as $n\rightarrow\infty$. This verifies the mixing relation in \cite[Equation (6)]{bai:2020:limit}  and hence also \cite[Assumption 2.1]{bai:2020:limit} after noting that $(b_n)$  here plays the same role as  $(\rho_n)$ in \cite{bai:2020:limit}. (Strictly speaking, $(\rho_n)$ in \cite{bai:2020:limit} was restricted to a subclass of $\RV_\infty(1-\beta)$, although this restriction can be easily relaxed.) In view of  \cite[Proposition 2.6]{bai:2020:limit}, we know that \cite[Assumption 2.1]{bai:2020:limit} holds if the system $(E,\cl{E},\mu, T,A)$   is replaced by  the product system $(E^p,\cl{E}^p,\mu^p,T_p,A^p)$. With \cite[Proposition 2.5]{bai:2020:limit} applied to the product system,   the conclusion of this lemma follows (or see directly   \cite[Equation (11)]{bai:2020:limit}).
\end{proof}
\subsection{Random measure and multiple stochastic integrals}\label{sec:mult int}  
%

We shall first provide a brief introduction to an infinitely divisible random measure.
For a formal definition, see \cite[Section 3.2]{samorodnitsky:2016:stochastic}.   
 Let $(E,\cl{E},\mu)$ be a  measure space where $\mu$ is $\sigma$-finite and atomless.   An infinitely divisible  (signed) random  measure $M$ on $(E,\cl{E})$ with control measure $\mu$ can be viewed as an infinitely divisible stochastic process  with index set $\cl{E}_0:=\{ B\in \cl{E}:\ \mu(B)<\infty\}$.   The random measure $M$ is \emph{independently scattered} and \emph{$\sigma$-additive}:  for any disjoint $B_n \in \cl{E}_0$, $n\in \bb{Z}_+$, the random variables $M(B_1),M(B_2),\ldots$ are independent, and if in addition $\cup_{n=1}^\infty B_n\in \cl{E}_0$, then  $M(\cup_{n=1}^\infty  B_n)=\sum_{n=1}^\infty M(B_n)$ a.s..  In view of these properties, the law of $M$ is completely specified by  the law of the marginal distributions of $M(B)$, $B\in \cl{E}_0$.  As an infinitely divisible process, it is well-known that $M$ can be decomposed into a Gaussian component and a Poisson component.
  We shall work with an infinitely divisible random measure $M$  with only  the Poisson component present. Below is the precise assumption on $M$ which defines the   stationary sequence in \eqref{eq:X_k intro}.
\begin{Ass}\label{ass:levy} 
$M$ is an infinitely divisible random measure         $(E,\cl{E},\mu)$    whose law is specified by
\begin{equation}\label{eq:chf M}
\E e^{iu M(B)}=\exp\left(-\mu(B) \int_{\bb{R}  } (1-\cos(uy))  \rho(dy)\right), \quad u\in \bb{R},
\end{equation}
where $B\in \cl{E}_0$, and    $\rho$ is a  symmetric  L\'evy measure on $\bb{R}$ satisfying $\rho(\{0\})=0$ with a unit second moment:
\begin{equation}\label{eq:2nd moment unit}
 \int_{\bb{R}} x^2 \rho(dx) = 1.
\end{equation} 
\end{Ass}
 These assumptions on $M$  are similar to those in \cite{bai:2020:functional}, except that in this paper $M$ has a finite variance. 
The symmetry of $\rho$ implies that \[\E M(B)=0,~B\in \cl{E}_0 \] and the standardization \eqref{eq:2nd moment unit}  implies
\begin{equation}\label{eq:2nd moment measure} 
E M(B)^2=-\frac{d^2}{du^2}\ln\E e^{iu M(B)}|_{u=0}=\mu(B)\int_{\bb{R}} x^2 \rho(dx)=\mu(B),~   B\in \cl{E}_0.
\end{equation}

We shall also need the following generalized inverse of the tail L\'evy measure: 
\begin{equation}\label{eq:rho inv}
\rho^{\leftarrow}(y)=\inf\{x> 0: \rho((x,\infty))\le y/2\},\quad y>0.
\end{equation}
The following relation between the generalized inverse and the moments of the L\'evy measure will be useful.
\begin{Lem}\label{Lem:rho moment}
We have for any $r>0$ that
     \begin{equation}\label{eq:levy moment equiv}
 \int_{\bb{R} } x^r \rho(dx)=  \int_0^\infty  \rho^{\leftarrow}(y)^r dy. 
\end{equation} 
\end{Lem} 
\begin{proof}
By the symmetry of $\rho$, Fubini and the equivalence   $y/2<\rho((x,\infty))\iff x<\rho^{\leftarrow}(y)$, $x,y>0$,  we have
 
\begin{align*}
 \int_{\bb{R}} x^r \rho(dx) &= 2 \int_{(0,\infty)} x^r \rho(dx)= 2\int_0^\infty r x^{r-1}  \rho ((x,\infty))dx=\int_{0}^\infty r x^{r-1}dx   \int_0^\infty 1_{\{y/2<\rho((x,\infty))\}} dy   \\
&= \int_0^\infty dy \int_0^{\rho^{\leftarrow}(y)}    r x^{r-1} dx =  \int_0^\infty  \rho^{\leftarrow}(y)^r dy.
\end{align*} 
\end{proof}

For a function $f\in L^2(E^p,\cl{E}^p,\mu^p)$, $p\in \bb{Z}_+$, the (off-diagonal) multiple integral  
\[
I_p(f)=\int_{E^p}' f(x_1,\ldots,x_p) M(dx_1)\ldots M(dx_p) 
\]
can be defined using a classical  approach orignated from \cite{ito:1951:multiple}: first for $f=1_{B_1\times\ldots\times B_p}$ with disjoint $B_i\in \cl{E}_0$,    $i=1,\ldots,p$, define $I_p(f)=M(B_1)\ldots M(B_p)$. Then  extend the definition to a linear combination of such functions, and finally   to a general $f\in L^2$ by a $L^2$ approximation. It is important, as indicated by the prime $'$, that the diagonal set $
D=\{(x_1,\ldots,x_p)\in E^p:\ x_i=x_j,\ i\neq j \}$ is excluded in the integration.  In view of this, one may always treat the integrand $f$ as   $f 1_{D^c}$.
See   \cite[Section 5.4]{peccati:taqqu:2011:wiener} for more details. Because of the invariance of the integral with respect to the permutation of the variables of $f$, one   can often assume without loss of generality that $f$ is symmetric, that is, its value is invariant with respect to any permutation of its variables.  For symmetric $f_1,f_2\in L^2(E^p,\cl{E}^p,\mu^p)$, we have the $L^2$ isometry property:
\begin{equation}\label{eq:L^2 iso}
\E I_{p_1}(f_1)I_{p_2}(f_2)= \begin{cases} 
p! \mu^p (f_1 f_2) &\text{ if }p_1=p_2=:p;\\
0 &\text{ if }p_1\neq p_2.
\end{cases}
\end{equation}

Alternatively, the multiple integral   may be constructed through  a series representation of the symmetric  infinitely divisible random measure $M$ without a Gaussian component (see, e.g., \cite{szulga:1991:multiple}). Such a construction is used in a coupling argument in Section \ref{sec:red}    and in the proof of tightness in $D[0,1]$ in Section \ref{sec:proof CLT}. If in addition $\rho(\bb{R})<\infty$, the multiple integral may also be expressed through a compound Poisson representation of $M$ (see Section \ref{sec:red}),  which we  shall use to facilitate   the computation of moments.

\section{Main results}\label{Sec:main}

In this section we state the main results.  
Throughout this section we shall make the following assumptions:
\begin{itemize}
\item $(E,\cl{E},\mu)$ is an atomless $\sigma$-finite infinite-measure space;
\item $(E,\cl{E},\mu,T)$ is ergodic and conservative    with a distinguished subspace $A$ satisfying Assumption \ref{ass:dyn}; 
\item The random measure  $M$ satisfies Assumption \ref{ass:levy};
 \item    The stationary sequence $(X_k)$ is as in \eqref{eq:X_k intro}, where $f$ is a symmetric bounded and $\mu^p$-a.e.\ continuous function on $A^p$ (extended to $(A^p)^c$ by taking zero value there). 
\end{itemize}
Note that the $\mu^p$-a.e.\ continuity of $f$ is with respect to the product  topology of the subspace $A$ in Assumption \ref{ass:dyn} .  Such $f$ is always in $L^2(\mu^p)$ since it is bounded and $\supp(f)\subset A^p$ where $\mu^p(A^p)=\mu(A)^p<\infty$, and hence $(X_k)$ is well-defined in view of Section \ref{sec:mult int}.

We first  clarify the memory property of $(X_k)$  implied by Assumption \ref{ass:dyn}. Recall the notation $\mu^p(f)=\int_{E^p} f d\mu^p$. 

\begin{Cor}\label{Cor:cov}
We have as $k\rightarrow\infty$,
\[
\E[X_k X_0]=p!  \mu^p\left((f\circ T_p^k) \cdot f\right)\sim p!  \mu^p(f)^2 b_k^{-p}\in \RV_\infty(p(\beta-1)).
\]
\end{Cor}
 Corollary \ref{Cor:cov} follows from \eqref{eq:L^2 iso} and Lemma \ref{Lem:Cov Asymp}.
Depending on whether $p(\beta-1)<-1$ or $p(\beta-1)>-1$, the covariance $\E[X_kX_0]$ is summable or not. This corresponds to a classical distinction between short-range dependence and long-range dependence. In particular, in the short-range dependence regime  $p(\beta-1)<-1$, it can be shown that $\Var(\sum_{k=1}^n X_k)$ scales linearly as $n\rightarrow\infty$, whereas in the long-range dependence regime $p(\beta-1)>-1$, the variance $\Var(\sum_{k=1}^n X_k)$ scales super-linearly. Hence the order of normalization in limit theorems need to be chosen differently in these two situations. See, e.g., \cite[Chapter 2]{pipiras:2017:long}, for more details. 

\begin{Thm}\label{Thm:CLT}
If    $p(\beta-1)<-1$ (so necessarily $p\ge 2$), then  as $n\rightarrow\infty$,
\begin{equation}
\left(\frac{1}{n^{1/2}}\sum_{k=1}^{\floor{nt}} X_k\right)_{t\in [0,1]} \ConvFDD \left(\sigma B(t) \right)_{t\in [0,1]},
\end{equation}
where $\ConvFDD$ stands for convergence of finite-dimensional distributions, $B$ is a standard Brownian motion, and
\[
\sigma^2=\sum_{k=-\infty}^\infty \E[ X_kX_0]=\sum_{k=-\infty}^\infty p!  \mu^p\left((f\circ T_p^k) \cdot f\right).
\] If in addition, $\int_{\bb{R} }   x^4  \rho(dx)<\infty$, then   $\ConvFDD$  can be replaced   by weak convergence in $D[0,1] $ with  the uniform metric. 
\end{Thm}

The proof of Theorem \ref{Thm:CLT} can be found in Section \ref{Sec:proofs}. We believe that for the convergence in $D[0,1]$, the assumption of a finite fourth moment  is an artifact of our proof technique and may be relaxed.    We also note that in the case $p(\beta-1)=-1$, we anticipate a central limit theorem similar to \eqref{eq:CLT}  to hold with Brownian motion as limit, although depending on the slowly varying factor  in \eqref{eq:bn_RV}, an additional slowly varying factor may appear in the normalization in \eqref{eq:CLT}. Treating the case  $p(\beta-1)=1$ requires some technical but non-essential modification of the proof below for Theorem \ref{Thm:CLT}, which   we shall omit in this paper. 

\begin{Thm}\label{Thm:nCLT}
If $p(\beta-1)\in (-1,0)$, and $p=1$ or $2$, then  as $n\rightarrow\infty$,
\begin{equation}
\left(\frac{1}{a_n}\sum_{k=1}^{\floor{nt}} X_k\right)_{t\in [0,1]} \Rightarrow  \left(\mu^p(f) Z_{p,\beta}(t) \right)_{t\in [0,1]},
\end{equation}
where $\Rightarrow$ stands for weak convergence in $D[0,1]$ with  the uniform metric, $\mu^p(f) =\int_{A^p} f d\mu^p$, the process  $Z_{p,\beta}$ is the standard Hermite process as in \eqref{eq:Herm proc} (i.e., the fractional Brownian motion if $p=1$, and the Rosenblatt process if $p=2$), the normalization sequence 
\begin{equation}\label{eq:a_n}
(a_n)=  \left(  \left(\frac{ 1 }{ (1-p(1-\beta)/2)(1-p(1-\beta)) p! }\right)^{1/2}  \frac{n}{b_n^{p/2}}\right) 
  \in  \RV_\infty(1-p(1-\beta)/2)
\end{equation}
  where   $(b_n)$ is as in \eqref{eq:uniform ret} and $1-p(1-\beta)/2\in (1/2,1)$.

\end{Thm}The proof of Theorem \ref{Thm:nCLT} can be found in Section \ref{Sec:proofs}.
We   mention that for $p=1$,     a result similar to Theorem   \ref{Thm:nCLT} has been considered in  \cite[Theorem 9.4.7]{samorodnitsky:2016:stochastic}. There  the dynamical system $(E,\cl{E},\mu,T)$ was constructed using the path space of a null-recurrent Markov chain and its infinite invariant measure.  The proof exploited the infinite divisibility of a  single stochastic integral.

The reason we can only include cases $p=1$ and $2$  in Theorem \ref{Thm:nCLT}  is because moment determinacy either ceases to hold or is unknown for the limit law when $p\ge 3$ (see \cite{slud:1993:moment}, and a simple explanation is that higher $p$ leads to heavier tails of a multiple Gaussian integral, see e.g., \cite[Theorem 6.12]{janson:1997:gaussian}). Hence our proof based on the method of moments cannot conclude the cases where $p\ge 3$. Nevertheless, we expect the following conjecture to hold. 
\begin{Con}\label{Con:nclt}
The conclusion of Theorem \ref{Thm:nCLT} continues to hold if $p\ge 3$.
\end{Con}
The appearance   of Hermite processes as non-central limits may be better physically understood in view of  the new representations of Hermite processes recently obtained in \cite{bai:2019:representations}, which involve the local time of intersecting stable regenerative sets. Although the moment calculation performed in this paper (Proposition \ref{Pro:moment nCLT} below) cannot conclude Conjecture \ref{Con:nclt},  yet it provides a compelling evidence. A conclusive proof for $p\ge 3$ may need to exploit the local time representations in   \cite{bai:2019:representations}. 

We also mention that it is natural to consider an extension of the results in the paper to the case where the infinitely divisible random measure $M$ has both   Gaussian   and   non-Gaussian components. Indeed when $p=1$, such an extension   is straightforward since one can   decompose the single integral into two independent components. When $p\ge 2$, however, such an independent decomposition no longer holds.  Hence the extension does not follow  from a straightforward combination of the results in \cite{bai:2020:limit} and those in the current work. This problem is left for a future work.
 
\section{Proofs of the main results}\label{Sec:proofs}

First  we provide a summary of the proof strategy. We first establish a reduction result which enables us to replace the   random measure $M$ in Assumption \ref{ass:levy}   by one with a finite L\'evy measure $\rho_0$   whose  moments of all orders exist.  This reduction result is justified through a coupling argument based on series representations of  infinitely divisible random measures  without   Gaussian components (see Lemma \ref{Lem:reduction} below). With such a finite L\'evy measure $\rho_0$, the corresponding random measure admits a compound Poisson representation with all the moments available. We can hence approach the convergence of finite-dimensional distributions in Theorems \ref{Thm:CLT} and \ref{Thm:nCLT} by the method of moments.  The tightness in $D[0,1]$ in Theorem \ref{Thm:CLT} is established via a fourth moment computation using the series representation. The tightness in Theorem \ref{Thm:nCLT}, on the other hand, follows from a well-known argument.

Below throughout, we use   $c$ to denote a generic positive constant, whose value may change from line to line.

\subsection{Reduction}\label{sec:red}

We shall follow the assumptions and notation   in Section \ref{Sec:main}.

Let $\rho_0$ be a symmetric L\'evy measure on $\bb{R}$ satisfying $\rho_0(\{0\})=0$ and
\begin{equation*}
 \int_{\bb{R} } (1\vee x^2) \rho_0(dx) <\infty.
\end{equation*} 
Then $\rho_0$ is integrable and hence the L\'evy measure of a compound Poisson distribution.
 Set 
\begin{equation}\label{eq:Q}
Q=\rho_0(\bb{R}).
\end{equation}

Recall that $A$ is the distinguished subspace in Assumption \ref{ass:dyn} .
 For fixed $n\in \bb{Z}_+$, we set
\[A_n=\bigcup_{k=1}^n T^{-k} A.
\]  
Note that in view of  \eqref{eq:w_n} we have $\mu(A_n)=w_n\in (0,\infty)$   since $\mu(A)\in (0,\infty)$.
On the probability measure space 
\begin{equation}\label{eq:prob measure spac}
(A_n,\cl{E}_n:=\cl{E}\cap A_n, \mu_n(\cdot):=\mu|_{\cl{E}_n}(\cdot )/w_n), 
\end{equation} 
where $\mu|_{\cl{E}_n}$ denotes the restriction of $\mu$ to $\cl{E}_n$, we define a random measure
\begin{equation}\label{eq:compound poisson M}
M_{1,n}(\cdot) =  \sum_{i=1}^{N_n}  Z_i  \delta_{U_{i,n}} (\cdot ) ,
\end{equation}
where $\delta_x$ is the delta measure at $x\ in E$, $(U_{i,n})$ are i.i.d.\ random elements taking value in $A_n$ with distribution $\mu_n$,   $(Z_i)$ are i.i.d.\ symmetric real random variables with distribution $\rho_0(\cdot)/Q$   whose moments of all orders exist,   
\[
N_n:=N(Q\mu(A_n))=N(Qw_n) 
\] 
with $(N(t))_{t\ge 0}$  being a unit-rate Poisson process, and $(U_{i,n})$, $(Z_i)$ and $(N(t))$  are all independent of each other.  Using some well-known properties of the Poisson process, one can verify that $M_{1,n}$ is infinitely divisible and independently scattered (and it is obviously $\sigma$-additive).  
  An elementary   computation  (see, e.g., \cite[Example 3.1.1]{samorodnitsky:2016:stochastic}) shows that the random measure   $M_{1,n}$  satisfies   \eqref{eq:chf M}  but with $\rho$   replaced by $\rho_0$.

Next, we introduce a second random measure  on $(A_n,\cl{E}_n)$ by setting
\begin{equation}\label{eq:M_2,n}
M_{2,n}(\cdot)= \sum_{i=1}^\infty \epsilon_i \rho_0^{\leftarrow}(\Gamma_i/w_n) \delta_{U_{i,n}}(\cdot ),
\end{equation}
where $(\epsilon_i)$ are i.i.d.\ Rademacher random variables, $\Gamma_i=E_1+\ldots+E_i$ with $(E_j)$   i.i.d.\ standard exponential random variables,   $(U_{i,n})_{i=1,\ldots,n}$ are i.i.d.\ random elements taking value in $A_n$  with distribution $\mu_n$ as before, $\rho_0^{\leftarrow}$ is the generalized inverse of $\rho_0$ as defined in \eqref{eq:rho inv},
 and $(\epsilon_i), (E_i)$ and $(U_{i,n})$ are independent. The  form  \eqref{eq:M_2,n}  is in general  known as a series representation of  an infinitely divisible process  (e.g., \cite{rosinski:1990:series}, \cite[Section 3.4]{samorodnitsky:2016:stochastic}). It follows from  \cite[Theorem 3.4.3]{samorodnitsky:2016:stochastic} (see also \cite{rosinski:samorodnitsky:1999:product}) that   $M_{2,n}$ on $(A_n,\cl{E}_n)$ is also an infinitely divisible random measure  satisfying \eqref{eq:chf M}    with $\rho$   replaced by $\rho_0$, and hence
\begin{equation}\label{eq:M_1,n=M_2,n}
M_{1,n}(\cdot)\EqD  M_{2,n}(\cdot) 
\end{equation}
for fixed $n\in \bb{Z}_+$.  

We   introduce a third random measure $M_{3,n}$ defined as in \eqref{eq:M_2,n} using the same   $(\epsilon_i)$ and $(\Gamma_i)$ and $(U_{i,n})$, except that $\rho_0^{\leftarrow}$ is replaced by $\rho^{\leftarrow}$. Then as above $M_{3,n}$ is an infinitely divisible random measure satisfying \eqref{eq:chf M} and hence 
\begin{equation}\label{eq:M_3,n=M}
M_{3,n}(\cdot)\EqD M(\cdot),
\end{equation} 
where $M$ is as in Assumption \ref{ass:levy}  but restricted to the subspace $(A_n,\cl{E}_n)$.

Here we explain the reason we introduce these   random measures. 
In particular,
for an integrand $f:E^p\rightarrow \bb{R}$ as described in Section \ref{Sec:main} which is symmetric, bounded and has support within $A^p$, we   introduce for $1\le k\le n$ that
\begin{equation}\label{eq:X_{k,1}}
X_{k,1}^{(n)}:=\int_{E^p}' (f\circ T_p^k)(x_1,\ldots,x_p) M_{1,n}(dx_1)\ldots M_{1,n}(dx_p)   =  p!\sum_{I\in\cl{D}_p(N_n)}  \left(\prod_{i\in I} Z_i\right) (f\circ T_p^k) (U_{I,n})   ,
\end{equation}
where   $U_{I,n}=(U_{i_1,n},\ldots, U_{i_p,n})$,  and
\begin{equation}\label{eq:D_p(n)}
\cl{D}_p(n):=\{I=(i_1,\ldots,i_p):\ 1\le  i_1<\ldots<i_p \le n\}.
\end{equation}
To obtain the second equality in
 \eqref{eq:X_{k,1}},  we have used the  exclusion of the diagonals   of the multiple integral and   the symmetry of $f$. Meanwhile, we set
  for $1\le k\le n$ that
\begin{equation}\label{eq:X_{k,2}}
X_{k,2}^{(n)}:=\int_{E^p}' (f\circ T_p^k)(x_1,\ldots,x_p) M_{2,n}(dx_1)\ldots M_{2,n}(dx_p)   =  p!\sum_{ I \in \cl{D}_p} \left(\prod_{i\in I}\epsilon_i  \rho^{\leftarrow}_0(\Gamma_i/w_n)\right)  (f\circ T_p^k)(U_{I,n}) ,
\end{equation}
where
\begin{equation}\label{eq:D_p}
\cl{D}_p:=\{I=(i_1,\ldots,i_p):\ 1\le  i_1<\ldots<i_p \}.
\end{equation}
The multilinear series in \eqref{eq:X_{k,2}} converges unconditionally a.s., namely, regardless of the order  the terms are added,  the series converges a.s.\ to the same limit (cf.\ \cite[Section 1]{rosinski:samorodnitsky:1999:product}). In view of \eqref{eq:M_1,n=M_2,n} we have 
\begin{equation}\label{eq:X_k,1=X_k,2}
\left(X_{k,1}^{(n)}\right)_{1\le k\le n}\EqD \left(X_{k,2}^{(n)}\right)_{1\le k\le n}.
\end{equation}
At last, we let $X_{k,3}^{(n)}$ be defined as \eqref{eq:X_{k,2}} but with $M_{2,n}$ replaced by $M_{3,n}$, and in view of \eqref{eq:M_3,n=M}, we have  
\begin{equation}\label{eq:X_k,3=X_k}
\left(X_{k,3}^{(n)}\right)_{1\le k\le n}\EqD \left(X_k\right)_{1\le k\le n},
\end{equation}
where $(X_k)$ is as in \eqref{eq:X_k intro}. The idea is that by approximating $\rho$   with $\rho_0$, one can then  approximate  $\left(X_{k,3}^{(n)}\right)_{1\le k\le n}$   with $\left(X_{k,2}^{(n)}\right)_{1\le k\le n}$   (see Lemma \ref{Lem:reduction} below for more details). This enables one to work with eventually $\left(X_{k,1}^{(n)}\right)_{1\le k\le n}$, whose form  is more amenable to the computation of moments compared to $\left(X_{k,2}^{(n)}\right)_{1\le k\le n}$.

Next, we prepare some results which are useful for the main reduction lemma  below. They will also be useful in the proofs by the method of moments later.
For $I\in \cl{D}_p$ and $t\in [0,1]$, we define
\begin{equation}\label{eq:L_n}
L_{n,I,t}= \sum_{k=1}^{\floor{nt}}(f\circ T_p^k)(U_{I,n}).
\end{equation}
Recall that when $\beta<1-1/p$, the sum $\sum_{k}\mu^{ p}((f\circ T_p^k)\cdot f) $ converges due to Lemma \ref{Lem:SRD sum var}.
\begin{Lem}\label{Lem:SRD sum var}
When $\beta<1-1/p$, we have for any $I\in \cl{D}_p$,
\begin{equation*} 
\E [L_{n,I,t_1} L_{n,I,t_2}]\sim  n w_n^{-p} (t_1 \wedge t_2)  \sum_{k=-\infty}^{\infty}  \mu^{ p}((f\circ T_p^k)\cdot f) 
\end{equation*}
as $n\rightarrow\infty$. In addition for any $ 1\le m\le n$,
\[
\E\left[\left(\sum_{k=1}^{m}(f\circ T_p^k)(U_{I,n})\right)^2\right]\le  w_n^{-p} m \sum_{k=-\infty}^{\infty}  |\mu^{ p}((f\circ T_p^k) f) |.
\]
\end{Lem}
\begin{proof}
For $-n\le k\le n$, set
\[
\gamma_{n}(k):=\E \left[ (f\circ T_p^k)(U_{I,n})\cdot  f(U_{I,n})\right]= w_n^{-p}\mu^{p}((f\circ T_p^k)\cdot f ).
\]
Assume without loss of generality that $0<t_1\le t_2\le 1$. We now focus on the first claim.
Using the invariance $\mu^p(T_p^{-1}\cdot )=\mu_p(\cdot )$, we have
\begin{align*}
\E [L_{n,I,t_1} L_{n,I,t_2}]&= \sum_{k_1=1}^{\floor{nt_1}} \sum_{k_2=1}^{\floor{nt_2}} \gamma_n(k_2-k_1)=\sum_{k_1=1}^{\floor{nt_1}} \sum_{k_2=1}^{\floor{nt_1}} \gamma_n(k_2-k_1) +\sum_{k_1=1}^{\floor{nt_1}} \sum_{k_2=\floor{nt_1}+1}^{\floor{nt_2}} \gamma_n(k_2-k_1)\\&=:\mathrm{I}_n+\mathrm{II}_n,
\end{align*}
where the   term $\mathrm{II}_n$ is understood as zero if $\floor{nt_1}=\floor{nt_2}$. By Lemma \ref{Lem:Cov Asymp}, we have $\mu^{ p}((f\circ T_p^k) \cdot f)\sim b_k^{-p} \mu^{ p}(f)^2 $ as $k\rightarrow\infty$ which belongs to $\RV_\infty(p(\beta-1))$ with  $p(\beta-1)<-1$, and hence $\sum_{k=-\infty}^{\infty}  |\mu^{ p}((f\circ T_p^k)\cdot f)|<\infty $.  It  then follows from  \cite[Lemma 5.4.4]{pipiras:2017:long} that $\mathrm{II}_n=w_n^{-p}o(n)$ as $n\rightarrow\infty$. In addition, by rearranging the double sum, 
\begin{align*}
\mathrm{I}_n =  n w_n^{-p}   t_1\sum_{k=-\floor{nt_1}}^{\floor{nt_1}} \left(\frac{\floor{n t_1}}{nt_1}- \frac{|k|}{nt_1}\right)\mu^{p}((f\circ T_p^k)\cdot f ).
\end{align*}
The first claim follows if one shows that the sum above converges to  $\sum_{k=-\infty}^{\infty}  \mu^{ p}((f\circ T_p^k)\cdot f) $ as $n\rightarrow\infty$. This can be verified by the dominated convergence theorem since    $\sum_{k=-\infty}^{\infty}  |\mu^{ p}((f\circ T_p^k)\cdot f)|<\infty $.

Now we turn to the second claim. For any $ 1\le m\le n$, we have
\begin{align*}
\E\left[\left(\sum_{k=1}^{m}(f\circ T_p^k)(U_{I,n})\right)^2\right] =  m\sum_{k=-m}^{m} \left(1-\frac{|k|}{m} \right) \gamma_n(|k|).
\end{align*}
The conclusion follows  since    $\left|\left(1-\frac{|k|}{m} \right) \gamma_n(|k|)\right|\le |\gamma_n(|k|)|$ for $-m\le k\le m$.

\end{proof}

  Introduce for $q\ge 2$, a symmetric function $h_q^{(\beta)} $  which is a.e.\ defined on $(0,1)^q$ as
\begin{equation}\label{eq:hq}
h_q^{(\beta)}(x_1,\dots, x_q) = \Gamma(\beta)\Gamma(2-\beta)\prod_{j=2}^q (x_j-x_{j-1})^{\beta-1},\  0<x_1<\cdots<x_q<1.
\end{equation}
Define also $h_0^{(\beta)}:=1$ and $h_0^{(\beta)}(x):=\Gamma(\beta)\Gamma(2-\beta)$. For $r\in \bb{Z}_+$, define
\begin{equation}\label{eq:cl I}
\cl{I}(i)=\{\ell\in \{1,\ldots,r\}:\ i\in I_\ell \}.
\end{equation}
\begin{Lem}\label{Lem:moments LRD}[\cite[Proposition 5.3]{bai:2020:functional}.] 
When $\beta\in (1-1/p,1)$, we have
for any $I_1,\ldots,I_r \in \cl{D}_p$, $t_1,\ldots,t_r\in [0,1]$ that
\begin{align}\label{eq:moment sum}
\lim_n \E  \prod_{\ell =1}^r \left(\frac{b_n^{p}}{n}  L_{n,I_\ell,t_{\ell}}\right)  = \mu^{p}(f)^{r} \int_{(\mbf{0},\mbf {t})} \prod_{i=1}^K  h_{|\cl{I}(i)|}^{(\beta)} (\mbf{x}_{\cl{I}(i)})  d\mbf{x}  
\end{align}
 where $(\mbf{0},\mbf {t})=(0,t_1)\times\ldots\times (0,t_r)$,   $\mbf{x}_{\cl{I}(i)}$ is the subvector of $\mbf{x}:=(x_1,\ldots,x_r)$  indexed by $\cl{I}(i)$   (because each $h_{q}^{(\beta)}$ is symmetric, the order of variables in $\mbf{x}_{\cl{I}(i)}$ does not matter), $d\mbf{x}=dx_1\ldots dx_r$, and
  $K = \max(\bigcup_{\ell=1}^r I_\ell)$.
\end{Lem}

We are now ready to state the  main reduction lemma.
\begin{Lem}\label{Lem:reduction}
Suppose the convergence  of finite-dimensional distributions in Theorem \ref{Thm:CLT} or \ref{Thm:nCLT} hold  for any    symmetric L\'evy measure $\rho=\rho_0$   so that  (recall \eqref{eq:levy moment equiv})
\begin{equation}\label{eq:rho_0 moments}  
\int_{\bb{R}} x^r \rho_0(dx) <\infty   \ \text{ for any }r\ge 0,\ \text{ and }   \int_{\bb{R}} x^2 \rho_0(dx)=\|\rho_0^{\leftarrow}\|_{L^2(\bb{R}_+)}^2=1.
\end{equation}
 Then the corresponding convergence  of finite-dimensional distributions also hold for general $\rho$ satisfying Assumption \ref{ass:levy}.
\end{Lem}
\begin{proof}
Fix the L\'evy measure $\rho$ as in Assumption \ref{ass:levy}. 
For any $\epsilon\in (0,1)$, there exists  a symmetric L\'evy measure $\rho_0=\rho_0^{(\epsilon)}$ satisfying \eqref{eq:rho_0 moments} such that (recall the generalized inverse in \eqref{eq:rho inv})
\begin{equation}\label{eq:approx rho_0 rho}
\|\rho^{\leftarrow}-\rho^{\leftarrow}_0 \|_{L^2(\bb{R}_+)}\le \epsilon.
\end{equation}
Indeed, it is not difficult to construct the desired $\rho^{\leftarrow}_0$    as a right-continuous non-increasing simple function with a  bounded support. 
Define
\[
W_{n,j}(t):=  \sum_{k=1}^{\floor{nt}} X_{k,j}^{(n)},\quad j=2,3,\quad t\in [0,1],
\]
where $X_{k,2}^{(n)}$ and $X_{k,3}^{(n)}$ are as introduced  in \eqref{eq:X_{k,2}}  corresponding to $\rho_0$ and $\rho$ respectively. Recall from \eqref{eq:X_k,3=X_k}  that $(X_{k,3}^{(n)},\ k=1,\ldots,n)\EqD (X_k,\ k=1,\ldots,n)$, the latter being the stationary sequence in Theorems \ref{Thm:CLT} or \ref{Thm:nCLT}. On the other hand, $(X_{k,2}^{(n)},\ k=1,\ldots,n)\EqD (X^{(\epsilon)}_k,\ k=1,\ldots,n)$, the latter being a stationary sequence defined using a multiple stochastic integral as   $(X_k,\ k=1,\ldots,n)$, but with the only difference that the L\'evy measure $\rho$ of the random measure $M$ is replaced by $\rho_0$. 
Then using the orthogonality induced by $\prod_{i\in I} \epsilon_i$ and independence, we have
\begin{align}
\E |W_{n,2}(t)- W_{n,3}(t)|^2&=  (p!)^2 \E \left| \sum_{ I \in \cl{D}_p} \left(\prod_{i\in I}\epsilon_i \right) \left(\prod_{i\in I}\rho_0^{\leftarrow}(\Gamma_i/w_n)- \prod_{i\in I}\rho^{\leftarrow}(\Gamma_i/w_n)\right) \left(\sum_{k=1}^{\floor{nt}} (f\circ T_p^k)(U_{I,n})\right)\right|^2 \notag\\ & =   (p!)^2   \left(\E \sum_{I\in \cl{D}_p}\left|\prod_{i\in I}\rho_0^{\leftarrow}(\Gamma_i/w_n)- \prod_{i\in I}\rho^{\leftarrow}(\Gamma_i/w_n)\right|^2\right) \E L_{n,I_0,t} ^2, \label{eq:coupling diff} 
\end{align}
where $I_0$ is an arbitrary fixed element of $\cl{D}_p$ and $L_{n,I_0,t}$ is as in \eqref{eq:L_n}. 
Set a sequence of $p$-variate functions as
 \[g_q(x_1,\ldots,x_p)=\prod_{i=1}^q \rho^{\leftarrow}(x_i) \prod
_{i=q+1}^p \rho^{\leftarrow}_0(x_i),\quad q=0,\ldots,p,  \]
where a product is understood as $1$ if the starting index exceeds the ending index.  Note that the sum in the first expectation  in \eqref{eq:coupling diff}  can be viewed as $\xi^p(h)/p!$, where $\xi^p(h):=\int_{\bb{R}_+^p}' h(x_1,\ldots,x_p)\xi(dx_1)\ldots \xi(dx_p) $ is an off-diagonal multiple integral on $\bb{R}_+^p$ of the function $h:=|g_0-g_p|^2$ with respect to the Poisson random measure $\xi:=\sum_{i=1}^\infty \delta_{\Gamma_i}$.   According to  \cite[Lemma 10.1(i)]{kallenberg:2017:random}, the expectation $\E \xi^p(h)=\int_{\bb{R}_+^p}  h(x_1,\ldots,x_p) dx_1\ldots  dx_p$. Applying this with 
 triangular inequalities, \eqref{eq:approx rho_0 rho} and the fact $\epsilon\le 1$, we have 
\begin{align*}
&p!\E \sum_{I\in \cl{D}_p}\left|\prod_{i\in I}\rho_0^{\leftarrow}(\Gamma_i/w_n)- \prod_{i\in I}\rho^{\leftarrow}(\Gamma_i/w_n)\right|^2=  w_n^p \|g_0-g_p\|_{L^2(\bb{R}_+^p)}^2\le  w_n^p  \left(\sum_{q=0}^{p-1} \|g_q-g_{q+1}\|_{L^2(\bb{R}_+^p)}\right)^2 \\&\le  w_n^p \left( p   \max(\|\rho_0^{\leftarrow}\|_{L^2(\bb{R}_+)},\|\rho^{\leftarrow}\|_{L^2(\bb{R}_+)} )^{p-1} \|\rho^{\leftarrow}-\rho^{\leftarrow}_0 \|_{L^2(\bb{R}_+)}\right)^2
\le  w_n^p p^2  \left(\|\rho^{\leftarrow}\|_{L^2(\bb{R}_+)}+1\right)^{2p-2}\epsilon^2.
\end{align*}
On the other hand,  by  Lemmas \ref{Lem:SRD sum var}, \ref{Lem:moments LRD} and  \eqref{eq:b_n w_n}, there exists a constant $c>0$ which does not depend on $n$ or $\epsilon$ that
\begin{align*}
  \E L_{n,I_0,t} ^2 \le 
  \begin{cases}
 c nw_n^{-p} &\text{ if } \beta<1-1/p,\\
 c n^2 w_n^{-2p}&  \text{ if } \beta>1-1/p.
  \end{cases}
\end{align*}
Hence returning  to \eqref{eq:coupling diff}, for some constant $c>0$ which does not depend on $n,\epsilon$ or $\rho_0$, we have
 \begin{align*}
 \E |W_{n,2}(t)- W_{n,3}(t)|^2 \le 
  \begin{cases}
c n \epsilon^2 &\text{ if } \beta<1-1/p,\\
 c n^2 w_n^{-p} \epsilon^2 &  \text{ if } \beta>1-1/p.
  \end{cases}
\end{align*}
Hence if $\beta<1-1/p$, we have  
\begin{equation}\label{eq:norm sum diff L2}
 \E |n^{-1/2}  W_{n,2}(t)- n^{-1/2} W_{n,3}(t)|^2\le c \epsilon^2.
\end{equation}
Now by assumption \[(n^{-1/2}  W_{n,3}(t))\ConvFDD (\sigma  B(t)),\] where $B(t)$ is a standard Brownian motion and \[\sigma^2=\sum_{k=-\infty}^\infty\E[X_kX_0]=\sum_{k=-\infty}^\infty \E[ X_k^{(\epsilon)} X_0^{(\epsilon)}].\] 
Note that here $\E[X_kX_0]=\E[ X_k^{(\epsilon)} X_0^{(\epsilon)}]$. This is because although $\rho_0$ can be different from $\rho$,  the same $L^2$ isometry relation \eqref{eq:L^2 iso}   holds for both   since they are standardized (see \eqref{eq:2nd moment unit} and \eqref{eq:rho_0 moments}).
  Hence the conclusion for the case $\beta<1-1/p$ follows from a well-known   approximation argument (e.g., \cite[Theorem 4.28]{kallenberg:2002:foundations}).
 The case $\beta>1-1/p$ is similar with the normalization  $n^{1/2}$ in \eqref{eq:norm sum diff L2} replaced by $a_n$ in \eqref{eq:a_n}. Note that $a_n\sim c n w_n^{-p}$ in view of \eqref{eq:b_n w_n} as $n\rightarrow\infty$

\end{proof}

In view of Lemma \ref{Lem:reduction} and \eqref{eq:X_k,1=X_k,2}, it suffices to prove the convergences of finite-dimensional distributions in Theorems \ref{Thm:CLT} and \ref{Thm:nCLT} with $(X_k)$ replaced by $(X_{k,1})$, the latter being defined by a compound Poisson random measure  with all moments finite.   This will be the objective  of Sections \ref{sec:proof CLT} and \ref{sec:proof NCLT} below.

\subsection{Proof of the central limit theorem}\label{sec:proof CLT}

Assume $p(\beta-1)<-1$.
Using $(X_{k,1}^{(n)})_{1\le k\le n}$ in \eqref{eq:X_{k,1}}, 
we define in this subsection
\begin{equation}\label{eq:S_n(t) CLT}
\left( S_n(t) \right)_{t\in [0,1]}:=\left( \frac{1 }{\sqrt{n}}\sum_{k=1}^{\floor{nt}}X_{k,1}^{(n)}\right)_{t\in [0,1]}=\left(   p! n^{-1/2}\sum_{I\in\cl{D}_p(N_n)} \left(\prod_{i\in I} Z_i\right) L_{n,I,t} \right)_{t\in [0,1]},
\end{equation}
where    $N_n=N(Qw_n)$ and
\[
L_{n,I,t}=\sum_{k=1}^{\floor{nt}} (f\circ T_p^k) (U_{I,n}).
\]
We need the following lemma when employing the method of moments.
\begin{Lem}\label{Lem:potter}
Under the assumptions in Section \ref{Sec:main},
for any $q\ge 2$ and any  $b\in (\beta-1,0)$,  there exists a constant $c>0$ which does not depend on $k_1,\ldots,k_q$,  such that
\[
\mu \left( \bigcap_{j=1}^q T^{-k_j} A \right)\le c (k_2-k_1)_1^{b}(k_3-k_2)_1^{b}\ldots (k_q-k_{q-1})^b_1
\]
for all  $1\le k_1\le\ldots \le k_q$,   
where   $(x)_1^{b}:=(x\vee 1)^b$.
\end{Lem}
\begin{proof}
We shall prove the conclusion by induction.
For $q=2$, using the measure-preserving property of $T$, we have
\[
\mu\left(  T^{-k_1}A \cap T^{-k_2} A \right)=\mu\left(  A \cap T^{-(k_2-k_1)} A \right).
\]
Recall that by Potter's bound for regular variation \cite[Theorem 1.5.6(i)]{bingham:goldie:teugels:1989:regular}, if a sequence $(a_n)\in \RV_\infty(-\gamma)$, $\gamma>0$, then for any $\gamma^*\in(0,\gamma)$, there exists some constant $c>0$ such that $a_n\le c n^{-\gamma^*}$.
Hence the conclusion follows from \eqref{eq:unif ret to mix} and Potter's bound. Now suppose that the conclusion holds for $q\ge 2$, and  we shall prove that it also holds for $q+1$. Indeed, using the measure preserving property of $T$ and the dual operator property \eqref{eq:dual},
\begin{align*}
\mu \left( \bigcap_{j=1}^{q+1} T^{-k_j} A \right)&=\mu \left( \bigcap_{j=1}^{q+1} T^{-(k_j-k_1)} A \right)=\int 1_A \times   \left(1_A \times  1_A \circ T^{k_3-k_2}  \times  \ldots  \times   1_A \circ T^{k_{q+1}-k_2} \right) \circ T^{k_2-k_1} d\mu\\
&=\int_A (\wh{T}^{k_2-k_1} 1_A )\times   \left(1_A \times  1_A \circ T^{k_3-k_2}  \times  \ldots  \times   1_A \circ T^{k_{q+1}-k_2} \right)   d\mu\\
&\le c  \mu(A) (k_2-k_1)^{b}_1 \mu \left( \bigcap_{j=2}^{q+1} T^{-k_j} A \right),
\end{align*}
for some constant $c>0$, where for the   inequality we have used  \eqref{eq:uniform ret} (note that $A$ is the whole subspace and thus $\mu$-a.e.\ continuous) and Potter's bound. Then the conclusion follows from the induction hypothesis.
\end{proof}
 
Now we are ready to carry out the method of moments computation.
\begin{Pro}\label{Pro:moment CLT}
Let $S_n(t)$ be as in \eqref{eq:S_n(t) CLT}  where $\rho_0$ defining $(X_{k,1}^{(n)})_{1\le k\le n}$ satisfies the assumptions in Lemma \ref{Lem:reduction}.
Assume $\beta<1-1/p$. Then
 as $n\rightarrow\infty$,
\begin{equation}\label{eq:moment conv CLT}
\E [S_n(t_1)\ldots S_n(t_r)]\rightarrow \sigma^r \E[ B(t_1)\ldots B(t_r) ]=\sigma^r \sum_{\cl{P}(r)} \prod_{j=1}^{r/2}  (t_{u_j}\wedge t_{v_j}),
\end{equation}
where
\[
\sigma^2:=\sum_{k=-\infty}^\infty \E[X_kX_0]=p!  \sum_{k=-\infty}^{\infty}  \mu^{ p}((f\circ T_p^k) \cdot f),
\]
  $B(t)$ is the standard Brownian motion,    and $\cl{P}(r)$ denotes the collection of all the partitions of $\{1,\ldots,r\}$  into disjoint pairs $\{u_j,v_j\}$, $j=1,\ldots,r/2$ if $r$ is even, and  is understood as $\emptyset$ (hence the last expression in \eqref{eq:moment conv CLT} is zero) if $r$ is odd.
\end{Pro}
\begin{proof} 
The   equality in \eqref{eq:moment conv CLT} follows from \cite[Theorem 1.28]{janson:1997:gaussian} and the covariance structure of a Brownian motion. So it is left to prove the   convergence in  \eqref{eq:moment conv CLT}. 

We use the following notation throughout to denote conditional expectation given the Poisson process $N$ in   \eqref{eq:X_{k,1}}: 
\[\E_N[\ \cdot \ ]=\E[ \ \cdot \ | N].\]
If $r=1$,   we have $\E [S_n(t_1)]=\E (\E_N[S_n(t_1)])=0$  due to the symmetry of $Z_i$, and hence \eqref{eq:moment conv CLT} holds. We  assume $r\ge 2$ throughout below. 

\medskip
\noindent \emph{Part 1}:  The first part of the proof aims at showing as $n\rightarrow\infty$,
\begin{equation}\label{eq:first part CLT}
\E_N [S_n(t_1)\ldots S_n(t_r)]\rightarrow \sigma^r \sum_{\cl{P}(r)} \prod_{j=1}^{r/2}  (t_{u_j}\wedge t_{v_j}) \quad\text{a.s..}
\end{equation}
   By independence, we have a.s.\
\begin{equation}\label{eq:moment S_n CLT}
\E_N [S_n(t_1)\ldots S_n(t_r)]=(p!)^r n^{-r/2}\sum_{I_1,\ldots,I_r\in\cl{D}_p(N_n)} \E \left[ \left(\prod_{i\in I_1} Z_i\right)\ldots  \left(\prod_{i\in I_r} Z_i\right)\right]\E\left(  \prod_{\ell =1}^r L_{n,I_\ell,t_{\ell}}\right).
\end{equation}
Due to the symmetry of the distribution of $(Z_i)$,  a factor $\E\left[ (\prod_{i\in I_1} Z_i)\ldots\left(\prod_{i\in I_r} Z_i\right)\right]\neq 0$  if and   only if the cardinality $|\cl{I}(i)|$ (recall $\cl{I}(i)$ defined in \eqref{eq:cl I}) is even for each $i=1,\ldots, N_n$, which can happen only if $pr$ is even.   If  $pr$ is odd, then so is $r$, and hence the limit in \eqref{eq:first part CLT} is zero. So \eqref{eq:first part CLT} trivially holds when $pr$ is odd.  We shall assume   $pr$ is even  below  throughout the proof of Part 1.   

We shall analyze  \eqref{eq:moment S_n CLT} by decomposing it into contributing and negligible terms. For this purpose, we introduce 
  for $m\ge p$,
\begin{equation}\label{eq:M(m)}
\cl{M}(m)=\{(I_1,\ldots,I_r):\  I_\ell \in \cl{D}_p(m),\ \ell=1,\ldots,r,\   |\cl{I}(i)| \text{ is even for }i=1,\ldots, m
  \},
\end{equation}
  where  $\cl{D}_p(m)$ is as in \eqref{eq:D_p(n)}.
When  $r$ is an even integer, we define 
\begin{equation}\label{eq:C(m)}
\cl{C}(m)=\{(I_1,\ldots,I_r)\in \cl{M}(m):\  \text{exactly $r/2$ pairs of $I_\ell$'s coincide   and different pairs are disjoint.}
  \}.
\end{equation}
\begin{figure} 
\centering
\begin{tabular}{|c|c|c|c|c|}
\hline 
$i$ & 1 & 2 & 3 & 4 \\ 
\hline 
$I_1$ & • &  &   &   • \\ 
\hline 
$I_2$ &  & • &  • &   \\ 
\hline 
$I_3$ & • &  &    & • \\ 
\hline 
$I_4$ &  & • &•   &   \\ 
\hline 
\end{tabular} 
\caption{\emph{The configuration above corresponds to $p=2$, $r=4$,   $I_1= I_3=\{1,4\}$, $I_2=I_4=\{2,3\}$. In this case, $\cl{I}(1)=\cl{I}(4)=\{1,3\}$, $\cl{I}(2)=\cl{I}(3)=\{2,4\}$.  So $(I_1,I_2,I_3,I_4)\in  \cl{C}(4)\cap \cl{N}(4)$.} }\label{ill:1}. 
\end{figure}
When $r$ is odd, set $\cl{C}(m)=\emptyset$.   Define for $m=p,p+1,\ldots,pr/2$ that
\begin{equation}\label{eq:N(m)}
\cl{N}(m)=\{(I_1,\ldots,I_r)\in \cl{M}(m):\     \cl{I}(i)\neq \emptyset,\ i=1,\ldots,m\}.
\end{equation}
 Note that we have suppressed in notation the dependence of $\cl{M}(m)$, $\cl{C}(m)$ and $\cl{N}(m)$ on $p$ and $r$. See Figures \ref{ill:1} and \ref{ill:2}  for  illustrations of the notation introduced above.

Since $N_n\uparrow \infty$ as $n\rightarrow\infty$  
a.s., we can assume without loss of generality that  $N_n \ge p$.  By the arguments below \eqref{eq:moment S_n CLT}, the index set $\cl{D}_p(N_n)$  under the summation sign in \eqref{eq:moment S_n CLT} can be replaced by $\cl{M}(\cl{N}_n)$. Decompose the sum in \eqref{eq:moment S_n CLT} into
\begin{equation}\label{eq:A}
A(n)=  (p!)^r n^{-r/2}  \sum_{(I_1,\ldots,I_r)\in\cl{C}(N_n)} \E  \left[ \left(\prod_{i\in I_1} Z_i\right)\ldots  \left(\prod_{i\in I_r} Z_i\right)\right]\E  \left(  \prod_{\ell =1}^r L_{n,I_\ell,t_{\ell}}\right)
\end{equation}
and
\begin{equation}\label{eq:B}
B(n)=  (p!)^r n^{-r/2} \sum_{(I_1,\ldots,I_r)\in\cl{M}(N_n)\setminus\cl{C}(N_n)} \E \left[ \left(\prod_{i\in I_1} Z_i\right)\ldots  \left(\prod_{i\in I_r} Z_i\right)\right]\E \left(  \prod_{\ell =1}^r L_{n,I_\ell,t_{\ell}}\right).
\end{equation}
Note that both $A(n)$ and $B(n)$ are stochastic since they depend on the Poisson count $N_n$. We shall show that $A(n)$ is the contributing term while $B(n)$ is negligible.

We     assume that $r$ is even so that $A(n)$ is possibly nonzero.
 To enumerate the elements in
  $\cl{C}(N_n)$,   first   select a partition from $\cl{P}(r)$ which specifies the pairings among $I_1,\ldots,I_r$. Then assign  ${N_n    \choose p}$ elements from $\{1,\ldots, N_n\}$ to the 1st pair, assign  ${  N_n  -p   \choose p}$ from the rest to the 2nd pair, $\ldots$, and assign  ${ N_n-pr/2 \choose p}$ from the rest to the last pair.   
Therefore, using  the fact that $Z_i$ follows the distribution $\rho_0(\cdot)/Q$ and the relation \eqref{eq:rho_0 moments}, we have
\begin{align*}
A(n)  = (p!)^r n^{-r/2}  \sum_{\cl{P}(r)}  {  N_n   \choose p}  {  N_n  -p   \choose p}\ldots  { N_n-pr/2 \choose p}     Q^{-pr/2}  \prod_{j=1}^{r/2} \E[ L_{n,I,t_{u_j}} L_{n,I,t_{v_j}}],
\end{align*}
where the sum is over partitions   $\{\{u_j,v_j\},\ j=1,\ldots,r/2\}\in\cl{P}(r)$ . 
Applying Lemma \ref{Lem:SRD sum var} and the fact $N_n\sim Qw_n$ a.s.,  we have a.s. 
\begin{align*} 
A(n)&\sim    n^{-r/2}     \sum_{\cl{P}(r)}   (Qw_n)^{pr/2}  Q^{-pr/2}   \prod_{j=1}^{r/2} \left(\left(\sum_{k=-\infty}^{\infty}  \mu^p((f\circ T_p^k) f)\right) (t_{u_j} \wedge t_{v_j}) n w_n^{-p}\right),
\end{align*}
 which  simplifies to  the right-hand side of \eqref{eq:first part CLT}.  
 
Next we show that $B(n)\rightarrow 0$ a.s.\ as $n\rightarrow\infty$.   First
  by H\"older's inequality, 
\begin{equation}\label{eq:holder}
\left|\E\left[ \left(\prod_{i\in I_1} Z_i\right)\ldots \left(\prod_{i\in I_r} Z_i\right)\right]\right|\le \E | Z_{1}|^{pr}<\infty .
\end{equation} Second,   the joint law of  $(L_{n,I_\ell,  t_{\ell}}, \ell=1,\ldots,r)$ is unchanged if the elements in $I_\ell$, $\ell=1,\ldots,r$, are replaced by the elements of $\{1,\ldots,m\}$, $m=|\cup_{\ell=1}^r I_\ell|$, based on any one-to-one correspondence.
Exploring these facts we have
\begin{equation}\label{eq:B(n) bound}
|B(n)|\le c n^{-r/2} \sum_{m=p}^{pr/2}  {N_n \choose m}    \sum_{(I_1,\ldots,I_r)\in \cl{N}(m)\setminus\cl{C}(m)}   \E\left(  \prod_{\ell =1}^r | L_{n,I_\ell,t_{\ell}}|\right).
\end{equation}

Fix for now $m$ and
  $(I_1,\ldots,I_r)\in \cl{N}(m)\setminus\cl{C}(m)$. Our next goal is to provide a bound for   $\E\left(  \prod_{\ell =1}^r | L_{n,I_\ell,t_{\ell}}|\right)$. 
It follows from a triangular inequality, the restriction $t_\ell\in [0,1]$  and the assumptions on $f$ that
\begin{equation}\label{eq:L simple bound}
|L_{n,{I_\ell},t_\ell}|\le c \sum_{k=1}^{n} (1_{A^p}\circ T_p^k) (U_{I_\ell,n}).
\end{equation}
Hence 
 \begin{align}\label{eq:bound prod L}
 \E\left(  \prod_{\ell =1}^r | L_{n,I_\ell,t_{\ell}}|\right)\le c \sum_{k_1,\ldots,k_r=1}^n
     \prod_{i=1}^m \int  \left(\prod_{\ell \in \cl{I}(i)} 1_A\circ T^{k_l}\right)   d\mu_n 
=  c \sum_{k_1,\ldots,k_r=1}^n \prod_{i=1}^m f_{|\cl{I}(i)|,n}(k_{\cl{I}(i)}),
\end{align}
where $\cl{I}(i)$ is as in \eqref{eq:cl I} and
\[ 
f_{q,n}: \{1,\ldots,n\}^q \rightarrow [0,1] ,\quad f_{q,n}(k_1,\ldots,k_q)=\begin{cases}
1 & \text{ if }q=0;\\
\mu \left( \bigcap_{\ell =1}^q T^{-k_\ell} A \right)/w_n& \text{ if }q\ge 1. \end{cases} 
\]
Note that since $f_{q,n}$ is symmetric, the order of the variables in $k_{\cl{I}(i)}$  does not matter.
By Lemma \ref{Lem:potter} and the measure-preserving property of $T$, we have
\[
f_{q,n}(k_1,\ldots,k_q)\le c \begin{cases}
1 & \text{ if }q=0;\\
w_n^{-1} & \text{ if }q=1;\\
w_n^{-1} (k_2-k_1)_1^{b} \ldots (k_q-k_{q-1})^b_1 & \text{ if }q\ge 2, 
\end{cases}
\]
where $b$ is chosen to satisfy  (recall $\beta-1<-1/p$)
 \begin{equation}\label{eq:b range}
(\beta-1)\vee\left(- \frac{1}{p-1} \right) <b<-\frac{1}{p}.
\end{equation}  
Next, we shall provide a bound for
\begin{align*}
 \sum_{1\le k_1\le \ldots \le k_r\le n} \prod_{i=1}^m f_{|\cl{I}(i)|,n}(k_{\cl{I}(i)}), 
\end{align*}
which  in turn yields a bound for the full sum in \eqref{eq:bound prod L} by adding up all $r!$ orders of $k_1,\ldots,k_r$.
Suppose $\cl{I}(i)=\{u(i,1),u(i,2)\ldots,u(i,|\cl{I}(i)|)\}\subset\{1,\ldots,r\}$, $u(i,1)<\ldots<u(i,|\cl{I}(i)|)$, $i=1,\ldots,m$. When $1\le k_1\le \ldots \le k_r\le n$, one has
\[
\prod_{i=1}^m f_{|\cl{I}(i)|,n}(k_{\cl{I}(i)})=w_n^{-m}\prod_{i=1}^m  \prod_{s=2}^{|\cl{I}(i)| }(k_{u(i,s)}-k_{u(i,s-1)})_1^b \le w_n^{-m}\prod_{i=1}^m  \prod_{s=2}^{|\cl{I}(i)|} (k_{u(i,s)}-k_{u(i,s)-1})_1^b,
\]
where the product $\prod_{s=2}^{|\cl{I}(i)| }$ is understood as $1$ if $|\cl{I}(i)|<2$, and   for the replacement of $u(i,s-1)$ by  $u(i,s)-1$ in the   inequality, we have used the monotonicity of the function $x\mapsto (x)_1^{b}$. Next, by rearranging the product, we have
\[
\prod_{i=1}^m  \prod_{s=2}^{|\cl{I}(i)|} (k_{u(i,s)}-k_{u(i,s)-1})_1^b = \prod_{u=2}^r  (k_{u}-k_{u-1})^{ d_u b}_1,
\]
where  
\[d_u=|\{1\le i\le m:\  u \in \cl{I}(i),\ \exists u'\in \cl{I}(i) \text{ s.t. } u'<u\}|,\quad u=2,\ldots,r.
\]
\begin{figure} 
\centering
\begin{tabular}{|c|c|c|c|}
\hline 
$i$ & 1 & 2 & 3 \\ 
\hline 
$I_1$ & • & • &   \\ 
\hline 
$I_2$   & • &   & • \\ 
\hline 
$I_3$ & •  & •  &  \\ 
\hline 
$I_4$ & • &  &   • \\ 
\hline 
\end{tabular} 
\caption{\emph{The configuration above corresponds to $p=2$, $r=4$,  $m=3$,  $I_1= I_3=\{1,2\}$, $I_2=I_4=\{1,3\}$. So  $(I_1,I_2,I_3,I_4)\in  \cl{N}(3)\setminus\cl{C}(3)$ and   $\cl{I}(1)=\{1,2,3,4\}$, $\cl{I}(2)=\{1,3\}$, $\cl{I}(3)=\{2,4\}$.   When $1\le k_1\le k_2\le k_3 \le k_4\le n$, we have $w_n^3\prod_{i=1}^3 f_{|\cl{I}(i)|,n}(k_{\cl{I}(i)})=(k_2-k_1)^b_1(k_3-k_2)^b_1 (k_4-k_3)^b_1  \times (k_3-k_1)^b_1 \times (k_4-k_2)^b_1\le    (k_2-k_1)^{b}_1 (k_3-k_2)^{2b}_1 (k_4-k_3)^{2b}_1.$}}\label{ill:2}. 
\end{figure}
See Figure \ref{ill:2} for an illustration of the bound of $\prod_{i=1}^m f_{|\cl{I}(i)|,n}(k_{\cl{I}(i)})$ above.
Let
$
L =|\{2\le u \le r:\  d_u<p \}|+1.$
Then 
\begin{align}\label{eq:key rel}
m=|I_1| + |I_2\setminus I_1|+\ldots+ |I_r\setminus(\cup_{j=1}^{r-1} I_j) |= p+(p-d_2)+\ldots+(p-d_r)=pL-\sum_{u=2}^r d_u 1_{\{d_u<p\}}.
\end{align} 
By \eqref{eq:b range}, we have  $d_u b<-1$ if $d_u=p$ and $d_ub>-1$ if $d_u<p$. So summing  iteratively in the order $k_r,k_{r-1},\ldots,k_1$, we obtain
\begin{align}\label{eq:bound off diag}
 \sum_{1\le k_1\le \ldots \le k_r\le n} \prod_{u=2}^r  (k_{u}-k_{u-1})^{ d_u b}_1 &\le c  n^{ \sum_{u=2}^r (1+ d_ub)1_{\{d_u<p\}}+1}= c n^{(pb+1)L-bm},
\end{align}
where for the last equality we have used the relation \eqref{eq:key rel} and the definition of $L$.
There are two cases: (i) $\sum_{u=2}^p d_u 1_{\{d_u<p\}}>0$ and (ii) $\sum_{u=2}^p d_u 1_{\{d_u<p\}}=0$. In case (i),  from \eqref{eq:key rel} we deduce $L\ge (m+1)/p$. Hence the exponent in the bound \eqref{eq:bound off diag} satisfies
\[
(pb+1)L-bm\le(pb+1)(m+1)/p-bm = \frac{m}{p} +  b+1/p < \frac{m}{p}\le \frac{r}{2},
\]
where the last inequality holds since $m\le pr/2$.
In case (ii), since each $d_u=0$ or $p$, the index sets $I_\ell$'s  either coincide or  disjoint.     Because also  $(I_1,\ldots,I_r)\in \cl{N}(m)\setminus \cl{C}(m)$, it is not difficult to see that $m\le pr/2-p$. Note also $L\ge m/p$ in view of \eqref{eq:key rel}.  Hence 
\[
(pb+1)L-bm\le (pb+1)m/p-bm =   \frac{m}{p}\le \frac{r}{2}-1 <\frac{r}{2}.
\]
Combining these two cases and returning to \eqref{eq:bound prod L}, we conclude  that for some constant $\delta>0$ which does not depend on $n$ or $m$,   
\begin{equation}\label{eq:B neg}
\E\left(  \prod_{\ell =1}^r |L_{n,I_\ell,t_{\ell}}|\right)\le  c  n^{r/2-\delta} w_n^{-m}
\end{equation}
for any   $(I_1,\ldots,I_r)\in \cl{N}(m)\setminus\cl{C}(m)$.
Plugging \eqref{eq:B neg} into \eqref{eq:B(n) bound} and using the fact  that  $N_n\sim Q w_n$  a.s.\ as $n\rightarrow\infty$,    we have $B(n)\rightarrow 0$ a.s..

 \medskip
 
\noindent\emph{Part 2}. The second part of the proofs aims at establishing the uniform integrability of\[\{\E_N [S_n(t_1)\ldots S_n(t_r)],\ n\ge 1\},\] which combined   with \eqref{eq:first part CLT} concludes the proof.
For this it suffices to show the uniform boundedness in $n$ of
\[
\E \left[ \E_N [S_n(t_1)\ldots S_n(t_r)]^2\right]\le \E  \left[\E_N[ S_n(t_1)^2\ldots S_n(t_r)^2]\right].
\]
Similarly as the way \eqref{eq:B(n) bound} is obtained, we have
\begin{align*}
\E_N[ S_n(t_1)^2\ldots S_n(t_r)^2]\le c  \sum_{m=p}^{pr}  {N_n \choose m}    \frac{1}{n^{r}} \sum_{(I_1,\ldots,I_{2r})\in \cl{N}(m)}    \E \left(  \prod_{\ell =1}^{2r} |L_{n,I_\ell,t_{\ell}}|\right),
\end{align*}
where $\cl{N}(m)$ is as in \eqref{eq:N(m)} but with $r$ replaced by $2r$.
We claim for each $(I_1,\ldots,I_{2r})\in \cl{N}(m)$ that
\begin{equation}\label{eq:bound L twice moment}
    \E \left(  \prod_{\ell =1}^{2r} |L_{n,I_\ell,t_{\ell}}|\right)\le c w_n^{-m} n^{r}.
\end{equation}
  Indeed, 
starting with the bound \eqref{eq:L simple bound},  this follows from  arguments similar to Part 1: introducing $\cl{C}(m)$ as \eqref{eq:C(m)} but with $r$ replaced by $2r$, dividing   $\cl{N}(m)$ into   $\cl{N}(m)\cap \cl{C}(m)$ and $\cl{N}(m)\setminus  \cl {C}(m)$ (note $\cl{C}(m)$ is nonempty only when  $m=pr$),     and then applying Lemma \ref{Lem:SRD sum var} and \eqref{eq:B neg} respectively in these two cases.
So
\[
\E_N[ S_n(t_1)^2\ldots S_n(t_r)^2]\le   c \sum_{m=p}^{pr} \frac{ N_n(N_n-1)\ldots (N_n-m+1)}{w_n^{m}}.
\]
Applying the fact $\E [N_n(N_n-1)\ldots (N_n-m+1)]=Q^m w_n^m$  concludes the desired boundedness.

\end{proof}

 
\begin{proof}[Proof of Theorem \ref{Thm:CLT}]
Recall for a random vector $(\xi_1,\ldots,\xi_d)$, $d\in \bb{Z}_+$, the set of its (multivariate) moments is  $\{E[\xi_1^{m_1}\ldots \xi_d^{m_d}]:\ m_1,\ldots,m_d \in \{0,1,2,\ldots\} \}$ given that all the expectations exist.  The distribution of $(\xi_1,\ldots,\xi_d)$ is said to be moment-determinate  if it is the only distribution with such a set of moments. It is known (see, e.g., \cite[Theorem 14.6]{schmudgen:2017:moment}) that  the moment determinancy of every marginal univariate distribution of $\xi_i$, $i=1,\ldots,d$, implies the moment determinancy of the joint distribution of  $(\xi_1,\ldots,\xi_d)$. In addition, it is well-known that if a limit  distribution (univariate or multivariate) is moment-determinate, then the convergence of all the moments implies the weak convergence to the limit distribution.
Suppose first that $\rho=\rho_0$ as in Lemma \ref{Lem:reduction}. Focusing on the joint distributions  at a fixed finite set of time points,  the convergence of all the  moments  follows  from Proposition \ref{Pro:moment CLT} (note that the time points $t_1,\ldots,t_r$ in Proposition \ref{Pro:moment CLT} are allowed to coincide.).  It is well-known that a univariate Gaussian distribution is moment-determinate, and hence so is a multivariate Gaussian distribution. So   the convergence of the finite-dimensional distributions holds in this case.   The extension from $\rho_0$ to general $\rho$ follows from Lemma \ref{Lem:reduction}. 

We are left to show the tightness in $D[0,1]$ with the uniform metric under the additional condition $\int_{\bb{R}}  x^4  \rho(dx)<\infty$, which by \eqref{eq:levy moment equiv} is equivalent to 
\begin{equation}\label{eq:4 th moment finite}
\int_{0}^\infty \rho^{\leftarrow}(x)^4dx<\infty.
\end{equation}  
Define 
\[
S_n'(t)= \frac{1 }{\sqrt{n}}\sum_{k=1}^{\floor{nt}}X_{k,3}^{(n)} =    p!   \sum_{ I \in \cl{D}_p} \left(\prod_{i\in I}\epsilon_i  \rho^{\leftarrow}(\Gamma_i/w_n)\right) \left( \frac{1}{\sqrt{n}} \sum_{k=1}^{\floor{nt}}  (f\circ T_p^k)(U_{I,n})\right),\quad t\in [0,1],
\]
which, in view of \eqref{eq:X_k,3=X_k}, has the same finite-dimensional distributions with the process defined as above but with $(X_{k,3}^{(n)})_{1\le k\le n}$ replaced by $(X_k)_{1\le k\le n}$.
For $0\le t_1< t_2\le 1$   and $I\in \cl{D}_p$,   define the measurable map
\begin{equation*} 
L_{n,t_1,t_2}: A_n^p \rightarrow \bb{R},\quad  L_{n,t_1,t_2}(u_1,\ldots,u_p):= \frac{1}{\sqrt{n}}\sum_{k=1}^{\floor{n t_2}-\floor{n t_1}} (f\circ T_p^k) (u_1,\ldots,u_p).
\end{equation*} 

Using the stationarity of $(X_{k,3}^{(n)})_{1\le k\le n}$ and applying  a generalized Khinchine inequality for multilinear forms in Rademacher random
variables (\cite[Theorem 1.3 (ii)]{samorodnitsky:szulga:1989:asymptotic})   conditioning on $(\Gamma_i)$ and $(U_{i,n})$,   we have
\begin{align}
\E [S_n'(t_2)-S_n'(t_1)]^{4}&=     \E\left[\left( p! \sum_{ I \in \cl{D}_p} \left(\prod_{i\in I}\epsilon_i  \rho^{\leftarrow}(\Gamma_i/w_n)\right) L_{n,t_1,t_2}(U_{I,n})\right)^4\right]\notag\\
&\le c \E\left[\left( \sum_{ I \in \cl{D}_p} \left( \prod_{i\in I}  \rho^{\leftarrow}(\Gamma_i/w_n)^2\right) L_{n,t_1,t_2}(U_{I,n})^2\right)^2\right].\label{eq:hypercontract}
\end{align} 
The sum inside the  square in \eqref{eq:hypercontract}  can be viewed as an off-diagonal multiple integral of the integrand
\[
  f_n(x_1,u_1,\ldots,x_p,u_p)=\left( \prod_{i=1}^p  \rho^{\leftarrow}(x_i/w_n)^2\right) L_{n,t_1,t_2}(u_1,\ldots,u_p)^2
\] on $(\bb{R}_+\times A_n)^p$ with respect to the (marked)  Poisson random measure $\sum_{i=1}^\infty \delta_{\Gamma_i,U_{i,n}}$ with intensity measure $\nu_n:=\lambda\times \mu_n$ (e.g., \cite[Lemma 12.2]{kallenberg:2002:foundations}), where $\lambda$ is the Lebesgue measure and $\mu_n$ is as in \eqref{eq:prob measure spac}.
 So  by     \cite[Lemma 10.1(iii)]{kallenberg:2017:random},  
 \begin{align*}
 \E [S_n'(t_2)-S_n'(t_1)]^{4} \le c \sum_{m=0}^p m!  {p\choose m}^2  \nu_n^{p-m}( \nu_n^m (f_n)^2),
 \end{align*}
 where $\nu_n^m (f_n)$ is understood as integrating  out $m$ of the $p$ variables of the symmetric function $f_n$ ($\nu^0_n f=f$), and  $\nu_n^{p-m}$ integrates out the $p-m$ variables left  in $\nu_n^m (f_n)^2$. Next,    it can be verified that
 \begin{align*}
\nu_n^{p-m}( \nu_n^m (f_n)^2 )=w_n^{p+m} \left(\int_{0}^\infty   \rho^{\leftarrow}(x)^4 dx\right)^{p-m} \E[L_{n,t_1,t_2}(U_{I_1,n})^2 L_{n,t_1,t_2}(U_{I_2,n})^2],
 \end{align*}
 where $I_1,I_2$ are arbitrary elements of  $\cl{D}_p$ satisfying $|I_1\cap I_2|=p-m$ (so   $|I_1\cup I_2|=p+m$).
Similarly as how \eqref{eq:bound L twice moment} is obtained,  using Lemma \ref{Lem:SRD sum var} and a bound as in \eqref{eq:B neg}  with $n$ replaced by $\floor{n t_2}-\floor{n t_1}$,  we have  
\begin{equation*}
\E[L_{n,t_1,t_2}(U_{I_1,n})^2 L_{n,t_1,t_2}(U_{I_2,n})^2]\le c w_n^{-|I_1\cup I_2|} \left(\frac{\floor{nt_2}-\floor{nt_1}}{n}\right)^2. 
\end{equation*} 
Combining these above we obtain
\begin{align*}
\E [S_n'(t_2)-S_n'(t_1)]^{4}&\le c  \left(\frac{\floor{nt_1}-\floor{nt_2}}{n}\right)^2,
\end{align*}
which concludes tightness in $D[0,1]$ in view of \cite[Lemma 4.4.1]{giraitis:koul:surgailis:2009:large}.

\end{proof}

\subsection{Proof of the non-central limit theorem}\label{sec:proof NCLT}

Now assume  $ p(\beta-1)\in (-1,0)$.
Using $(X_{k,1}^{(n)})_{1\le k\le n}$  in \eqref{eq:X_{k,1}}, we define   in this subsection 
\begin{equation}\label{eq:S_n(t) nCLT}
\left( S_n^*(t) \right)_{t\in [0,1]}=\left( \frac{b_n^p }{w_n^{p/2} n}\sum_{k=1}^{\floor{nt}}X_{k,1}^{(n)}\right)_{t\in [0,1]}\EqD\left(   p!w_n^{-p/2}\sum_{I\in\cl{D}_p(N_n)} \left(\prod_{i\in I} Z_i\right) L_{n,I,t}^*\right)_{t\in [0,1]},
\end{equation}
where    $N_n=N(Qw_n)$ and
\[
L_{n,I,t}^*= \frac{b_n^p}{n}\sum_{k=1}^{\floor{nt}} (f\circ T_p^k) (U_{I,n}).
\]
Note that $\left(\frac{b_n^p }{w_n^{p/2} n}\right)\in  \mathrm{RV}_\infty(p(1-\beta)/2-1)$, where $p(1-\beta)/2-1\in (-1,-1/2)$.
\begin{Pro}\label{Pro:moment nCLT}
Let $S_n^*(t)$ be as in \eqref{eq:S_n(t) nCLT}  where $\rho_0$ defining $(X_{k,1}^{(n)})_{1\le k\le n}$ satisfies the assumptions in Lemma \ref{Lem:reduction}.
Assume $\beta\in (1-1/p,1)$. Then as $n\rightarrow\infty$,
\begin{equation}\label{eq:moment conv main}
\E [S_n^*(t_1)\ldots S_n^*(t_r)]\rightarrow \E[ H_{p,\beta}(t_1)\ldots H_{p,\beta}(t_r) ],
\end{equation}
where 
$
H_{p,\beta}(t)$ is a constant multiple  of the  standard Hermite process $Z_{p,\beta}(t)$ in \eqref{eq:Herm proc}.
\end{Pro}
\begin{proof}
Let $E_N$ denote the conditional expectation given $N$.
We have
\begin{equation}\label{eq:moment S_n}
\E_N [S_n^*(t_1)\ldots S_n^*(t_r)]=(p!)^r w_n^{-rp/2}\sum_{I_1,\ldots,I_r\in\cl{D}_p(N_n)} \E \left[ \left(\prod_{i\in I_1} Z_i\right)\ldots  \left(\prod_{i\in I_r} Z_i\right)\right]\E \left(  \prod_{\ell =1}^r L_{n,I_\ell,t_{\ell}}^*\right).
\end{equation}
Assume without loss of generality that  $N_n\ge pr/2$. 
Similarly as the arguments below \eqref{eq:moment S_n CLT}, we can assume that $pr$ is even. Recalling $\cl{M}(m)$ in \eqref{eq:M(m)} and $\cl{N}(m)$ in \eqref{eq:N(m)}, we have
\begin{align}\label{eq:moment S_n cond}
\E_N [S_n^*(t_1)\ldots S_n^*(t_r)]&= (p!)^r w_n^{-pr/2}\sum_{(I_1,\ldots,I_r)\in\cl{M}(N_n)} \E\left[\left(\prod_{i\in I_1} Z_i\right)\ldots  \left(\prod_{i\in I_r} Z_i\right)\right]\E\left(  \prod_{\ell =1}^r L_{n,I_\ell,t_{\ell}}^*\right)\notag\\
&=
\sum_{m=p}^{pr/2}  (p!)^r w_n^{-pr/2} {N_n\choose m}\sum_{(I_1,\ldots,I_r)\in\cl{N}(m)} \E\left[ \left(\prod_{i\in I_1} Z_i\right)\ldots  \left(\prod_{i\in I_r} Z_i\right)\right]\E\left(  \prod_{\ell =1}^r L_{n,I_\ell,t_{\ell}}^*\right)\notag\\&=:\sum_{m=p}^{pr/2} T_{m}(n),
\end{align}
where  the second equality above follows from an argument similar to  the one leading to \eqref{eq:B(n) bound}.  

We first show that 
\[
\lim_n T_m(n)= 0 \quad \text{ for }  m=p,\ldots ,pr/2-1.
\]
 For this purpose, note first  that $\E [ \prod_{\ell =1}^r L_{n,I_\ell,t_{\ell}}^* ]$ are uniformly bounded with respect to $0\le t_\ell \le 1$ and $I_\ell$, $\ell=1,\ldots,r$. Indeed, this can be seen by bounding $|f|$ with a constant multiple of $1_{A^p}$, replacing $t_\ell$'s by $1$, and applying Lemma \ref{Lem:moments LRD} with $f=1_{A^p}$. In addition, $\E\left( (\prod_{i\in I_1} Z_i)\ldots (\prod_{i\in I_r} Z_i)\right)$ are also uniformly bounded in view of \eqref{eq:holder}.  Therefore,   
\begin{align}\label{eq:bound lower order}
T_m(n)\le c w_n^{-pr/2} {N_n \choose m},  \quad m=p,\ldots,pr/2.
\end{align}
The right-hand side of \eqref{eq:bound lower order}  tends to $0$  a.s.\ if $m<pr/2$
since $N_n/w_n\rightarrow Q $ a.s..  

Now we treat the leading term $m=pr/2$.  Note that   a configuration $(I_1,\ldots,I_r)$ belongs to $\cl{N}(pr/2)$ if and only if  $|\cl{I}(i)|=2$ (recall \eqref{eq:cl I} and \eqref{eq:N(m)}) for all $i=1,\ldots,pr/2$.  Hence by Lemma \ref{Lem:moments LRD} and \eqref{eq:rho_0 moments},
\begin{align*}
T_{pr/2}(n)&=(p!)^r w_n^{-pr/2}{N_n\choose pr/2} \sum_{(I_1,\ldots,I_r)\in  \cl{N}(pr/2)} \E\left[ \left(\prod_{i\in I_1} Z_i\right)\ldots  \left(\prod_{i\in I_r} Z_i\right)\right]\E\left(  \prod_{\ell =1}^r L_{n,I_\ell,t_{\ell}}^*\right)\\
&\rightarrow \frac{(p!)^r Q^{pr/2}}{(pr/2)!}  \mu^{p}(f)^{r}  \sum_{(I_1,\ldots,I_r)\in \cl{N}(pr/2)} \int_{(\mbf{0},\mbf {t})} \prod_{i=1}^{pr/2}  h_{2}^{(\beta)} (\mbf{x}_{\cl{I}(i)})  d\mbf{x} \quad   a.s.,
\end{align*}
We claim that the summation above equals
\[
2^{-pr/2}  [
 \Gamma(\beta)\Gamma(2-\beta)]^{pr/2}\sum \int_0^{t_1}ds_1\ldots\int_0^{t_r}ds_r |s_{u(1)}-s_{v(1)}|^{\beta-1}\ldots |s_{u(pr/2)}-s_{v(pr/2)}|^{\beta-1},
\]
where the sum above is over all indices $u(1),v(1),\ldots,u(pr/2),v(pr/2)\in \{1,\ldots,r\}$ such that $u(1)\neq v(1)$,\ldots, $u(pr/2)\neq v(pr/2)$ and each number $1,\ldots,r$ appears exactly $p$ times in $u(1),v(1),\ldots,u(pr/2),v(pr/2)$. To see this,   write $\cl{I}(i)=\{u(i),v(i)\}$, $i=1,\ldots,pr/2$,  and note that the factor $2^{-pr/2} $ above accounts for the ignorance of the order within each  pair $(u(i),v(i))$ in $\cl{I}(i)$.

Combining these above and returning   to \eqref{eq:moment S_n cond}, we get as $n\rightarrow\infty$ that
\begin{align*}
&\E_N [ S_n^*(t_1)\ldots S_n^*(t_r)]\\\rightarrow &\frac{ [\mu^p(f)   p!  ]^r [
 \Gamma(\beta)\Gamma(2-\beta)]^{pr/2}}{2^{pr/2} (pr/2)! }\sum \int_0^{t_1}ds_1\ldots\int_0^{t_r}ds_r |s_{u(1)}-s_{v(1)}|^{\beta-1}\ldots |s_{u(pr/2)}-s_{v(pr/2)}|^{\beta-1}.
\end{align*}
This is the joint $r$-th moment formula for a  (non-standardized) Hermite process \cite[Remark 4.4.2]{pipiras:2017:long} ($p$ there corresponds to    $r$ here,  and   $k$ there correspond to $p$ here.)

Now we are left to take another expectation  in \eqref{eq:moment S_n}. The conclusion will follow if the uniform integrability of \[(\E_N [S_n^*(t_1)\ldots S_n^*(t_r)],\ n\ge 1)\] holds. To show this, we consider the boundedness of
\[
\E \left[ \left(\E_N [S_n^*(t_1)\ldots S_n^*(t_r)]\right)^2\right]\le \E  \left[\E_N[ S_n^*(t_1)^2\ldots S_n^*(t_r)^2]\right].
\]
In view of \eqref{eq:moment S_n cond} and \eqref{eq:bound lower order},   we have
\begin{align*}
\E_N[ S_n^*(t_1)^2\ldots S_n^*(t_r)^2]&\le  c  w_n^{-2q} \sum_{m=p}^{pr} {N_n\choose m}.
\end{align*}
Using the fact $\E_N [N_n(N_n-1)(N_n-m+1)]=Q^m w_n^m$,
we see that 
$\E\left[ \E_N [S_n^*(t_1)\ldots S_n^*(t_r)]^2\right]$ is bounded.

\end{proof}

\begin{proof}[Proof of Theorem \ref{Thm:nCLT}]
Note that the  normalization $(a_n)$ chosen in \eqref{eq:a_n} ensures that the limit  has variance $\mu^p(f)$ at $t=1$, which can be verified with     Corollary \ref{Cor:cov}  and \cite[Proposition 2.2.5]{pipiras:2017:long}. Note also that in view of the relation between $(b_n)$ and $(w_n)$ in \eqref{eq:b_n w_n}, the normalization $\left(\frac{w_n^{p/2} n}{ b_n^p }\right)$ in Proposition \ref{Pro:moment nCLT} and the normalization $(a_n)$ in \eqref{eq:a_n} are asymptotically equivalent up to a constant.  
So in the case $\rho=\rho_0$ as in Lemma \ref{Lem:reduction}, the convergence of finite-dimensional distributions follows from Proposition \ref{Pro:moment nCLT},  the moment determinancy of a distribution living in the first and second Wiener chaos \cite{slud:1993:moment},  and the fact that the multidimensional moment determinacy follows from the moment determinacy of the marginals  (\cite[Theorem 14.6]{schmudgen:2017:moment}, see also the proof of Theorem \ref{Thm:CLT} above).  The extension from $\rho_0$ to general $\rho$ follows from Lemma \ref{Lem:reduction}.  
 Tightness in $D[0,1]$ is  a routine result in this long-range dependence  regime (Corollary \ref{Cor:cov}, \cite[Proposition 2.2.5]{pipiras:2017:long}  and \cite[Proposition 4.4.2]{giraitis:koul:surgailis:2009:large}).
 
\end{proof}

 \bigskip
 \noindent \textbf{Acknowledgment}. The author would like to thank Takashi Owada and Yizao Wang for helpful discussions.  The author would like to thank the anonymous referees for their careful reading and helpful suggestions which have lead to substantial improvements of the paper.
 
 \bigskip 
 
\noindent 
Shuyang Bai\\
Department of Statistics\\ 
University of Georgia\\
310 Herty Drive, \\
Athens, GA, 30602, USA. \\
{bsy9142@uga.edu}

\bibliographystyle{abbrvnat}
\bibliography{Bib}
\end{document}